\documentclass[]{amsart}

\usepackage{amsmath}
\usepackage{amsthm}
\usepackage{amssymb,amsfonts}
\usepackage{mathtools}
\usepackage{color}
\usepackage{esint}

\usepackage[all,arc]{xy}
\usepackage{tikz}
\usetikzlibrary{shapes,snakes,arrows, chains, positioning, shapes.geometric, shapes.symbols,calc}

\def\centerarc[#1](#2)(#3:#4:#5)
    { \draw[#1] ($(#2)+({#5*cos(#3)},{#5*sin(#3)})$) arc (#3:#4:#5); }

\newtheorem{theorem}{Theorem}[section]

\newtheorem{corollary}[theorem]{Corollary}

\newtheorem{lemma}[theorem]{Lemma}
\newtheorem{proposition}[theorem]{Proposition}
\newtheorem{remark}[theorem]{Remark}
\newtheorem*{ack}{Acknowledgments}
\numberwithin{equation}{section}

\newcommand{\C}{\mathbf{C}}
\newcommand{\R}{\mathbf{R}}

\newcommand{\mH}{\mathcal{H}}
\newcommand{\dd}{\partial \bar{\partial}}
\newcommand{\p}{\partial}
\newcommand{\cone}{\C_{\beta_1} \times \C_{\gamma}}
\newcommand{\coneb}{\C_{\beta} \times \C_{\gamma}}
\newcommand{\tz}{\tilde{z}}
\newcommand{\tw}{\tilde{w}}
\newcommand{\vol}{\mbox{vol}}
\newcommand{\mR}{\mathcal{R}}

\title[Three lines]{Calabi-Yau metrics with cone singularities along intersecting complex lines: the unstable case}
\author[M. de Borbon]{Martin de Borbon}
\address{King's College London}
\email{martin.deborbon@kcl.ac.uk}
\author[G. Edwards]{Gregory Edwards}
\address{University of Notre Dame}
\email{gedward2@nd.edu}

\begin{document}

\maketitle

\begin{abstract}
	We produce local Calabi-Yau metrics on \(\C^2\) with conical singularities along three or more complex lines through the origin whose cone angles strictly violate the Troyanov condition. The tangent cone at the origin is a flat K\"ahler cone with conical singularities along two intersecting lines: one with cone angle corresponding to the line with smallest cone angle, while the other forms as the collision of the remaining lines into a single conical line. Using a branched covering argument, we can construct Calabi-Yau metrics with cone singularities along cuspidal curves with cone angle in the unstable range.
\end{abstract}

\section{Introduction}

The geometry of Calabi-Yau metrics has received considerable attention since the seminal work of Yau \cite{yau}. In recent years such metrics, and the associated complex Monge-Amp\`ere equation, have been studied, for instance, on complete non-compact manifolds \cite{ChYa80,TiYau90,TiYau91,CoRo19,Li19,sz}, with singular volume form \cite{EyGuZe08,EGZ09,HeSu17,Ko98,Ko03}, and with conical singularities \cite{CGP,dBoSpo,GP,GuoSongI,GuoSongII}. The study of K\"ahler-Einstein metrics with conical singularities has been developed mainly in the normal crossing set-up. In this paper we investigate the more general not normal crossing case in the simplest situation of three or more complex lines meeting at a point.

We work on \(\C^2\) with standard complex coordinates \(z, w\). For $d\geq 3$, let \(L_1, \ldots, L_d\) be distinct complex lines through the origin with defining linear equations \(L_j = \{\ell_j =0\}\). 
Let \(0< \beta_1 \leq \beta_2 \leq \ldots \leq \beta_d < 1 \) satisfy
\begin{equation} \label{anglecond}
	(1-\beta_2) + \ldots + (1-\beta_d) < (1- \beta_1) .
\end{equation}
Write \(\C_{\beta}\) for the complex numbers endowed with the line element \(|z|^{\beta-1}|dz|\), which represents a cone of total angle \(2\pi\beta\). We now state our main result.

\begin{theorem}\label{MAINTTHM}
	Suppose $0 < \beta_1 \leq  \ldots \leq  \beta_d < 1$ satisfy Equation \eqref{anglecond}. Then there is a K\"ahler metric \(\omega_{CY}\) on a neighbourhood of \( 0 \in \C^2\) with the following properties: 
		\begin{enumerate}
			\item[(a)] it has cone angles \(2\pi\beta_j\) along \(L_j \setminus \{0\}\) for \(j=1, \ldots, d\);
			\item[(b)] it solves the Calabi-Yau equation \[\det(\omega_{CY}) = \prod_{j=1}^{d}|\ell_j|^{2\beta_j-2}\]
			in a neighbourhood of the origin;
			\item[(c)] its tangent cone at \(0\) is isometric to \( \C_{\beta_1} \times \C_{\gamma} \) where \(0< \gamma<1\) is determined by
				\begin{equation*}
					(1-\gamma) = (1-\beta_2) + \ldots + (1-\beta_d) .
				\end{equation*}
		\end{enumerate}	 
\end{theorem}

Some clarifying remarks are in order.
In item (a) we mean standard cone singularities in transverse directions, as considered by Donaldson \cite{Donaldson}. The Calabi-Yau equation in item (b) implies that \(\omega_{CY}\) is smooth and Ricci-flat on the complement of the lines. The tangent cone statement in item (c) is understood in the Gromov-Hausdorff sense.

The condition ~\eqref{anglecond} can be interpreted in terms of strict instability. The case where the inequality in \eqref{anglecond} is strictly violated, or equivalently, when the Troyanov condition (see \cite{LuTi92,tr91}) 
\[
	\sum_{j=1}^d (1 - \beta_j) > 2 \max_i \hspace{2pt} (1-\beta_i)
\]
holds was considered by de Borbon-Spotti \cite{dBoSpo} and modelled on polyhedral K\"ahler cones. 
The construction in the semistable case, where equality is obtained in \eqref{anglecond}, remains open. See Section \ref{sect:GENPICT} for further discussion.

Combining our main result together with a branched covering, we obtain the following result.

\begin{corollary}\label{cor:CUSP}
	Let \(C=\{u^m=v^n\}\) be a cuspidal curve with \(2 \leq m < n\) and let 
	\begin{equation}\label{eq:unstable}
		1 + \frac{1}{n} - \frac{1}{m} < \beta < 1 .
	\end{equation}
	There is a Calabi-Yau metric on a neighbourhood of \(0 \in \C^2\), with cone angle \(2\pi\beta\) along \(C \setminus \{0\}\) and tangent cone \(\C_{\tilde{\gamma}} \times \C\) at the origin with \(1-\tilde{\gamma}= m (1-\beta)\).
\end{corollary}

Indeed, we take \(\omega_{CY}\) in Theorem \ref{MAINTTHM} with \(\beta_1=1/n\) along \(\{w=0\}\), \(\beta_2=1/m\) along \(\{z=0\}\) and \(\beta_3=\beta\) along \(\{z=w\}\). The metric in Corollary \ref{cor:CUSP} is the pullback of \(\omega_{CY}\) by the branched covering map \((u, v) \to (u^m, v^n)\). 

The set of cone angles ~\eqref{eq:unstable} represents the `unstable range'. For cone angles in the `stable range',
\[
	1 - \frac 1 m - \frac 1 n < \beta < 1 + \frac 1 n - \frac 1 m,
\]
Calabi-Yau metrics were constructed in ~\cite{dBoSpo}. The existence of Calabi-Yau metrics with prescribed behaviour as in Corollary \ref{cor:CUSP} was speculated for the \(m=2\), \(n=3\) case in \cite[p. 213]{CDS15b}.

\subsection*{Outline}
For the sake of definiteness, and to simplify notation, we assume from now on that the number of lines in Theorem \ref{MAINTTHM} is \(d=3\). The arguments for \(d>3\) are the same as for the \(d=3\) case, with the obvious modifications.

In Section \ref{sect:FMETRIC}, we use the Green's function for the Laplacian to obtain potentials for flat metrics on $\C$ which have two cone singularities, with angles $2\pi\beta_2$ and $2 \pi\beta_3$ at each point, and are asymptotic to $\C_\gamma$ at infinity. In Section \ref{sect:APROXSOL}, we write down an approximate metric on $\C^2$, with small Ricci potential, which has the prescribed conical singularities along the lines and the desired tangent cone at the origin. The approximate metric is modelled on $\C_{\beta_1} \times \C_\gamma$ in a neighbourhood of the origin, and uses Sz\'ekelyhidi's ansatz \cite{sz} to glue in scaled copies of the flat metrics with two cone points asymptotic to $\C_\gamma$.

The main technical work of the paper is in Section \ref{sect:SCHAUDER} deriving an appropriate Schauder estimate for the Laplace operator of the approximate solution acting on H\"older spaces. The estimate is proved by Campanato iteration using harmonic approximations, along the lines of \cite{dBoEdw}. There are two key ingredients: (i) approximation of balls, up to a fixed error, by balls centred at the apex of suitable model cones; (ii)  \(C^{\alpha}\) control on the complex Hessian for a family of reference functions which approximate the subquadratic harmonic polynomials on the model cones. Once the Schauder estimate is proved, the approximate solution is perturbed to an actual Ricci-flat metric by means of a standard application of the implicit function theorem in Section \ref{sect:PERTURBATION}. Finally, in Section \ref{sect:GENPICT} we discuss the relation of the Calabi-Yau metrics to algebro-geometric notions of stability and higher dimensions.

\begin{ack}
	The authors would like to thank G\'abor Sz\'ekelyhidi for sharing with us his insights and helpful conversation.
	
	The first named author was financially supported by the Agence Nationale de la Recherche, project CCEM: \textbf{ANR-17-CE40-0034}. 
	
	The second named author was supported by the National Science Foundation RTG:
Geometry and Topology at the University of Notre Dame, grant number \textbf{DMS-1547292}.

\end{ack}

\section{Flat metrics on \(\C\) with two cone points} \label{sect:FMETRIC}

We set the defining equations to be \(L_1 = \{z=0\}\) and \( L_{j} = \{w = a_j z \}\) for \(j=2, 3\). The linear maps of \(\C^2\) that preserve \(L_1\) act on the slice \(\{z=1\} \cong \C\) by affine transformations. Performing a suitable linear change of coordinates we assume that the weighted centre of mass \(\sum_j (1-\beta_j)a_j\) is located at zero:
\begin{equation*}
(1-\beta_2)a_2 + (1-\beta_3)a_3 = 0 .
\end{equation*}

The area form
\begin{equation} \label{flatmetriconc}
\omega_F = \gamma^2 |w-a_2|^{2\beta_2 -2} |w - a_3|^{2\beta_3 -2} i d w d\bar{w} 
\end{equation}
defines a flat K\"ahler metric on \(\C\) with two conical singularities of angles \(2\pi \beta_2\) at \(w=a_2\) and \(2\pi \beta_3\) at \(w=a_3\). It is asymptotic to the cone \(\gamma^2 |w|^{2\gamma-2} i dw d \bar{w}\) at infinity, as follows by noticing that \(2\gamma-2 = (2\beta_2-2) + (2\beta_3 -2)\).

\begin{proposition}\label{propmetriconC}
	Assume that \((1-\beta_2)a_2 + (1-\beta_3)a_3 =0\), then we can solve 
	\begin{equation} \label{flateq}
	i \dd \phi = \omega_F
	\end{equation}
	with \(\phi\) asymptotic to \(|w|^{2\gamma}\) at infinity. More precisely, outside a compact set we can write
	\begin{equation} \label{fmpotential}
	\phi =  |w|^{2\gamma} + A \log |w| + \phi_0
	\end{equation}
	with
	\begin{equation*}
	A = \frac{\gamma^2}{\pi} \int_{\C} (|t-a_2|^{2\beta_2 -2} |t - a_3|^{2\beta_3 -2} - |t|^{2\gamma-2})  i dt d \bar{t}
	\end{equation*}
	and \(\phi_0 = O(|w|^{-c})\) with derivatives for \(c = \min \{2-2\gamma, 1\}\).
\end{proposition}

\begin{proof}
	We solve the corresponding Poisson equation by taking convolution with the Green's function
	\begin{equation} \label{flatpotential}
	\phi(w) =  |w|^{2\gamma} + \frac{\gamma^2}{\pi} \int_{\C} (|t-a_2|^{2\beta_2 -2} |t - a_3|^{2\beta_3 -2} - |t|^{2\gamma-2}) \log (|w - t|) d\mu(t),	
	\end{equation}
	where \(d\mu(t) = i dt d \bar{t}\) is the standard Lebesgue measure. Write
	\begin{equation*}
		f(t) = |t-a_2|^{2\beta_2 -2} |t - a_3|^{2\beta_3 -2} - |t|^{2\gamma-2} .
	\end{equation*}
	Our assumption \(a_2 (1-\beta_2) + a_3(1-\beta_3) =0\) implies that for \(|t|\gg 1\) we have
	\begin{align*}
	\left|	f(t) \right| &= |t|^{2\gamma-2} \left| |(1-a_2/t)^{\beta_2-1} (1-a_3/t)^{\beta_3-1}|^2 -1 \right| \\
	&=|t|^{2\gamma-2} \left| |(1-a_2 (\beta_2-1) /t + O (t^{-2})) (1-a_3(\beta_3-1)/t + O(t^{-2}))|^2 -1 \right| \\ 
	&= O(|t|^{-2-\epsilon}) ,
	\end{align*}
	with \(\epsilon=2-2\gamma>0\). Hence, the integral on the r.h.s. of equation \eqref{flatpotential} is convergent and \(\phi\) solves equation \eqref{flateq}. 
	
	For \(|w| \gg 1\) we can write
	\begin{equation} \label{fmpot}
		\phi(w) - |w|^{2\gamma} - A \log |w| = \frac{\gamma^2}{\pi} \int_{\C} f(t) \log |1-t/w| d\mu(t) .
	\end{equation}
	We estimate the integral on the r.h.s. of equation \eqref{fmpot} by dividing it into three regions:
	
	\begin{enumerate}
		\item[(i)] \(\{|t| \geq 2|w|\}\). We use that \(|f(t)| \leq C |t|^{-2-\epsilon}\) together with \( 0 \leq \log |1 - t/w| \leq C \log |t/w|\) to get
		\begin{align*}
		\int_{2|w|}^{\infty} r^{-1-\epsilon}  \log (r/|w|) dr &= |w|^{-\epsilon} \int_{2}^{\infty} s^{-1-\epsilon} (\log s) ds \\
		&= C |w|^{-\epsilon} .
		\end{align*}
		
		\item[(ii)] \(\{|w|/2 \leq |t| \leq 2|w|\}\). We use that \(|f(t)| \leq C |w|^{-2-\epsilon}\) to get 
		\begin{align*}
		\int_{|w|/2 \leq |t| \leq 2|w|} & |w|^{-2-\epsilon}  |\log |1 - t/w|| d\mu(t) \\
		&\leq |w|^{-2-\epsilon} \int_{|t-w| \leq 3|w|} \log (|t-w|/|w|) d\mu(t) \\
		&= |w|^{-2-\epsilon} 2\pi \int_0^{3|w|} |\log(r/|w|)| r dr \\
		&= |w|^{-\epsilon} 2\pi \int_0^{3} |\log s| s ds \\
		&= C |w|^{-\epsilon} .
		\end{align*}
		
		\item[(iii)] \(\{|t| \leq |w|/2\}\). Let \(R= \max \{|a_2|, |a_3|\}\). We use that \(|f(t)| \leq C |t|^{-2-\epsilon}\) if \(|t| \geq 2R\) together with \( |\log |1 - t/w| | \leq C |t/w|\) to get
		\begin{align*}
		\int_{|t| \leq |w|/2} |f(t)|  |t/w|  d\mu(t) &\leq C |w|^{-1} \left(\int_{|t| \leq 2R} |tf(t)| d\mu(t) + \int_{2R}^{|w|/2} r^{-\epsilon}   dr \right) \\
		&= O(|w|^{-\min \{1, \epsilon\} }) .
		\end{align*}
		
	\end{enumerate} 
	
	The above three estimates together imply that 
	\begin{equation} \label{fmpot2}
		\left|\int_{\C} f(t) \log |1-t/w| d\mu(t)\right| = O(|w|^{-\min \{1, \epsilon\}}) .
	\end{equation}
	Equation \eqref{fmpotential} follows from equations \eqref{fmpot} and \eqref{fmpot2}.
\end{proof}

The metric defined by equation \eqref{flatmetriconc} is obtained by doubling a truncated wedge on the Euclidean plane with interior angles \(\pi\beta_2\) and \(\pi\beta_3\).  If we let \(a_2, a_3\) to be real and take a branch of the logarithm, then the Schwarz-Christoffel integral (see \cite[Chapter 6]{ahlfors})
\begin{equation*}
F(w) = \int_0^w (t - a_2)^{\beta_2 -1} (t - a_3)^{\beta_3 -1} dt
\end{equation*}
gives a conformal equivalence between the upper half plane and the truncated wedge; see Figure \ref{fig:truncated wedge}. From this point of view, a potential for the metric is given by \(|F(w)|^2\).

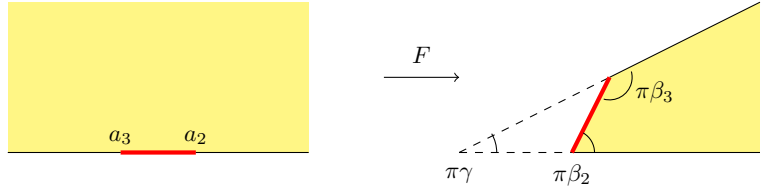
\begin{figure}
	\centering
	\begin{tikzpicture}
	
	\fill[yellow!60] (-2, 0) rectangle (2, 2);
	\fill[yellow!60] (5.5, 0) -- (8, 0) -- (8, 2) -- (6, 1) ;
	
	\draw[dashed] (4, 0) to (5.5,0);
	\draw[] (5.5, 0) to (8,0);
	\draw[dashed] (4, 0) to (6,1);
	\draw[] (6, 1) to (8, 2);
	\draw[red, ultra thick] (-.5, 0) to (.5,0);
	\draw[red, ultra thick] (5.5, 0) to (6,1);
	\draw[] (.5, 0) to (2,0);
	\draw[] (-.5, 0) to (-2,0);
	\draw[->] (3, 1) to (4,1);
	
	\centerarc[](5.5,0)(0:60:.3);
	\centerarc[](6,1)(-105:15:.3);
	\centerarc[](4,0)(0:30: .5);
	
	\draw  (-.5, 0.2) node (asd) [scale=.9] {\(a_3\)};
	\draw  (.5, .2) node (asd) [scale=.9] {\(a_2\)};
	\draw  (4, -0.3) node (asd) [scale=.9] {\(\pi \gamma\)};
	\draw  (5.5, -.3) node (asd) [scale=.9] {\(\pi \beta_2\)};
	\draw  (6.6, .8) node (asd) [scale=.9] {\(\pi \beta_3\)};
	\draw  (3.5, 1.3) node (asd) [scale=.9] {\(F\)};
	
	\end{tikzpicture}
	\caption{The Schwarz-Christoffel map uniformizes the truncated wedge. The flat metric \(\omega_F\) is obtained by doubling. If we parallel move the red segment of the truncated wedge, we realize the collision of the cone points \(2\pi\beta_2\) and \(2\pi\beta_3\) into \(2\pi\gamma\).}
	\label{fig:truncated wedge}
\end{figure}

Standard regularity theory implies that  \(\phi\) is smooth on \(\C \setminus \{a_2, a_3\}\) and it extends continuously over the cone points \(\{a_2, a_3\}\).

\begin{lemma} \label{lem:LOCREGFM}
	In a neighbourhood of \(a_2\) we can write 
	\begin{equation} \label{locregFm}
	\phi = e^{S_1} |w-a_2|^{2\beta_2} + S_2
	\end{equation}
	with \(S_1\) and \(S_2\)  smooth harmonic  functions; and similarly  around \(a_3\).
\end{lemma}

\begin{proof}
	Around the cone point \(a_2\) we can find a complex coordinate \(v=v(w)\) with \(v(a_2) = 0\) and such that \(\omega_F = \beta_2^2 |v|^{2\beta_2-2} i dv d\bar{v}\), see \cite[Lemma 3.4]{troyanovpolaires}. In particular, \(i\dd (\phi(v) - |v|^{2\beta_2}) = 0 \) and by removable of singularities \(\phi(v) - |v|^{2\beta_2}\) must extend smoothly over zero. Going back to the \(w\) coordinate we recover \eqref{locregFm}. 
\end{proof}

We differentiate equation \eqref{locregFm} to obtain the following.

\begin{corollary}
	Around \(a_2\) we have \(|\p \phi|= O(|w-a_2|^{2\beta_2-1})\) and similarly for \(a_3\). In particular, the ratio 
	\[\left| \frac{\p \phi}{\p w} \right|^2 \left(\frac{\p^2 \phi}{\p w \p \bar{w}}\right)^{-1}\]
	is uniformly bounded around the singularities.
\end{corollary}

\section{The approximate solution} \label{sect:APROXSOL}

\subsection{The ansatz} 
Write 
\begin{align*}
	r &= |z|^{\beta_1} , \\
	R &= |w|^{\gamma} , \\
	\rho^2 &= r^2 + R^2 .
\end{align*}
So \(\rho\) measures the distance to the origin with respect to \(\cone\). Similarly, \(r\) measures the distance to the \(\beta_1\)-line \(\{z=0\}\) and \(R\) measures the distance  to the \(\gamma\)-line \(\{w=0\}\) which  will be realized as the collision of \(L_2\) and \(L_3\).

Note that, since \(\gamma = \beta_2 + \beta_3 -1\), our crucial assumption given by equation \eqref{anglecond} is equivalent to
\begin{equation*}
	\gamma > \beta_1 .
\end{equation*}

Fix \(\alpha_0 \in (1, \gamma/\beta_1) \). 
We use standard cutoff functions such that \(\chi_1 + \chi_2 = 1 \), with \(\chi_1(s)=1\) if \(s>2\) and \(\chi_2(s)=1\) if \(s< 1\). Let \(\phi\) be the potential for the flat metric with two cone points given by Proposition \ref{propmetriconC}.
The approximate solution is 
\begin{equation}
\omega = i \dd \left( |z|^{2\beta_1} + \chi_1 (Rr^{-\alpha_0}) |w|^{2\gamma} + \chi_2 (Rr^{-\alpha_0}) |z|^{2\gamma} \phi (z^{-1}w) \right) .
\end{equation}

We now state the main result of this section.

\begin{proposition}\label{propaproxsol}
	\(\omega\) defines a K\"ahler metric on a neighbourhood of \( 0 \in \C^2\) with the following properties:
	\begin{enumerate}
		\item  it has standard cone singularities of angle \(2 \pi \beta_j \) along \(L_j \setminus \{0\}\);
		\item the tangent cone of \(\omega\) at \(0\) is \( \C_{\beta_1} \times \C_{\gamma} \).
	\end{enumerate}	 
\end{proposition}

The rest of this section is dedicated to the proof of Proposition \ref{propaproxsol}. We restrict to smaller neighbourhoods of the origin whenever necessary. Let's begin by analysing the K\"ahler potential of the approximate solution,
\[\psi = |z|^{2\beta_1} + \chi_1 |w|^{2\gamma} + \chi_2 |z|^{2\gamma} \phi(z^{-1}w) . \]

\begin{lemma}
	The potential \(\psi\) is smooth on the complement of the three lines and it extends continuously over them.
\end{lemma}

\begin{proof}
	If \(\Lambda>0\) is such that \(|a_2|, |a_3| < \Lambda\) then we can use the expansion \eqref{fmpotential} on the region \( |w|>\Lambda |z|\) and write
	\begin{align*}
	\psi &= |z|^{2\beta_1} + |w|^{2\gamma} + \chi_2 |z|^{2\gamma} \left(A \log |z^{-1}w| + \phi_0 \right) \\
	&= |z|^{2\beta_1} + |w|^{2\gamma} + O(|z|^{2\gamma}(-\log |z|) ) \\
	&= O(\rho^2) .
	\end{align*}
	On the other hand, if \(|w| \leq \Lambda |z|\) then \(\phi(z^{-1}w)\) is uniformly bounded  and therefore \(\psi =  O(|z|^{2\beta_1})\) on this region. We conclude that \(\psi = \rho^2 +\) (higher order terms) around \(0\). In particular \(\psi\) extends continuously over the origin with \(\psi(0)=0\). 
	
	Let \(s=Rr^{-\alpha_0}\) and \(\kappa = \alpha_0 \beta_1/\gamma\), so \(\beta_1/\gamma < \kappa < 1\). On  \(\C^2 \setminus \{0\}\), we have:
	\begin{itemize}
		\item \(\psi = |z|^{2\beta_1} + |w|^{2\gamma}\) if \(s>2\), that is \(|w|>2^{1/\gamma}|z|^{\kappa}\).
		\item \(\psi\) is smooth if \( 1/2 < s < 4 \).
		\item \(\psi = |z|^{2\beta_1} + |z|^{2\gamma}\phi(z^{-1}w)\) if \(s<1\), that is \(|w|<|z|^{\kappa}\).
	\end{itemize}
	We see that \(\psi\) extends continuously over \(\{z=0\}\) by \(\psi=|w|^{2\gamma}\). Since \(\phi: \C \to \R\) is continuous and smooth on \(\C \setminus \{a_2, a_3\}\), we see that \(\psi\) extends continuously over \(L_2\) by \(\psi|_{L_2} = |z|^{2\beta_1} + \phi(a_2)|z|^{2\gamma}\) and similarly for \(L_3\). Note that, because \(\phi\) is smooth at \(0 \in \C\), then \(\psi\) is extends smoothly over \(\{w=0\}\) outside the origin. 
\end{proof}

The restriction of \(\omega_{\cone}\) to a complex line \(L=\{w=az\}\) agrees with \(i \dd (|z|^{2\beta_1} + |a|^{2\gamma}|z|^{2\gamma})\) and has cone angle \(2\pi\beta_1\) at the origin; while its restriction to \(\{z=0\}\) agrees with \(\C_{\gamma}\). The approximate solution \(\omega\) exhibits this same behaviour 
Given a complex line \(L=\{w=az\}\), we have \(\omega|_L = i \dd (|z|^{2\beta_1}+\phi(a)|z|^{2\gamma})\) provided \(|z|\) is sufficiently small but how small depends on the line. 

Next we want to show that \(\omega\) is indeed a positive form. We begin by analysing the region that contains the colliding lines \(L_2\) and \(L_3\). Recall that \(R=|w|^{\gamma}\) and \(r=|z|^{\beta_1}\).

\begin{lemma} \label{lem:COLLREG}
	There is \(\delta>0\) such that 
	\begin{equation} \label{aproxsolV}
		i\dd \left(|z|^{2\beta_1} + |z|^{2\gamma}\phi(z^{-1}w)\right) >0
	\end{equation} 
	on \(\{R < \delta r\}\) and it is uniformly equivalent to the Hermitian metric 
	\begin{equation} \label{eq:HERMMET}
		\omega_H = \beta_1^2 |z|^{2\beta_1-2} i dz d\bar{z} + \gamma^2 |w-a_2z|^{2\beta_2-2}|w-a_3z|^{2\beta_3-2} idwd\bar{w}.
	\end{equation}
\end{lemma}

\begin{proof}
	Let \(\xi=w/z\). On the complement of \(\{z=0\}\) we use the map 
	\[(z, \xi) \to (z, w=z\xi)\] 
	and pull-back the form \eqref{aproxsolV}. Recall that
	\[\omega_F = \gamma^2 |\xi-a_2|^{2\beta_2-2}|\xi-a_3|^{2\beta_3-2}id\xi d\bar{\xi} .\]
	We will show that \eqref{aproxsolV} is uniformly equivalent to the warped product Hermitian metric 
	\[ \omega_\mathrm w = \omega_{\C_{\beta_1}} + |z|^{2\gamma} \omega_F = i\eta_1 \wedge \bar{\eta}_1 + i \eta_2 \wedge \bar{\eta}_2 , \]
	with
	\[\eta_1 = \beta_1 |z|^{\beta-1} dz, \hspace{2mm} \eta_2 = \gamma |z|^{\gamma} |\xi-a_2|^{\beta_2-1} |\xi-a_3|^{\beta_3-2}d\xi . \]
	
	In the coordinates \((z, \xi)\), we have
	\begin{equation} \label{eq:collreg1}
		i \dd \left(|z|^{2\beta_1} + |z|^{2\gamma} \phi (\xi) \right) 
		= \left(\omega_{\C_{\beta_1}} + \phi(\xi)\gamma^2|z|^{2\gamma-2}idzd\bar{z}\right) + |z|^{2\gamma} \omega_F + E .	
	\end{equation}
	First, note that
	\[\big|\phi(\xi)|z|^{2\gamma-2}idzd\bar{z}|_{\omega_\mathrm w}=O(|z|^{2\gamma-2\beta_1}|\xi\big|^{2\gamma}).\]
	
	The off diagonal term on the r.h.s. of equation \eqref{eq:collreg1} is
	\begin{align*}
	E &= \frac{\p^2}{\p z \p \bar{\xi}} \left( |z|^{2\gamma} \phi (\xi) \right) idz d\bar{\xi} + \mbox{(conjugate)} \\
	&= B i\eta_1 \wedge \bar{\eta}_2 + \mbox{(conjugate)} 
	\end{align*}
	where 
	\begin{align*}
		|B| &= |z|^{2\gamma-1} \left|\frac{\p \phi}{\p \xi}\right| |z|^{1-\beta_1} |z|^{-\gamma} |\xi-a_2|^{1-\beta_2} |\xi-a_3|^{1-\beta_3} \\
		&= |z|^{\gamma-\beta_1} \left( \left|\frac{\p \phi}{\p \xi}\right| |\xi-a_2|^{1-\beta_2} |\xi-a_3|^{1-\beta_3} \right).
	\end{align*}
	By Lemma \ref{lem:LOCREGFM}, the function \(|\p \phi / \p \xi| |\xi-a_2|^{1-\beta_2} |\xi-a_3|^{1-\beta_3}\) is uniformly bounded in a neighbourhood of the singularities \(\xi= a_2, a_3\). Furthermore, using that \(|\p \phi / \p \xi |= O(|\xi|^{2\gamma-1})\) as \(|\xi| \to \infty\), together with \(\gamma-1 = (\beta_2-1) + (\beta_3-1)\), we find that there is  \(C>0\) such that for all \(\xi \in \C\)
	\[\left|\frac{\p \phi}{\p \xi}\right| |\xi-a_2|^{1-\beta_2} |\xi-a_3|^{1-\beta_3} \leq C (|\xi|^{\gamma} + 1) .\]
Hence
\[
	|B| \leq C |z|^{\gamma-\beta_1} (|\xi|^{\gamma} +1) .
\]

	We conclude that if
	\begin{equation} \label{regioncondition}
		|\xi|^{\gamma} < (1/2C) |z|^{\beta_1-\gamma} ,
	\end{equation}
	then \(i \dd \left(|z|^{2\beta_1} + |z|^{2\gamma} \phi (\xi) \right)\) is uniformly equivalent to \(\omega_\mathrm w\).\footnote{Note that, since \(\gamma-\beta_1>0\), if \(|z|\) is sufficiently small, say \(|z|^{\gamma-\beta_1} < 1/(4C)\), then $|B|< 3/4$ is uniformly bounded.} Finally, replacing \(\xi=w/z\), we see that Equation \eqref{regioncondition} is equivalent to \(|w|^{\gamma}<(1/2C)|z|^{\beta_1}\). That is, \(R<\delta r\) with \(\delta=1/2C\).	
	
	A similar computation shows that, on the region \(\{R<\delta r\}\), 
	\[ | (\omega_{\C_{\beta_1}} + |z|^{2\gamma} \omega_F) - \omega_H |_{\omega_H} = O(r^{-1}R) \]
	which proves the claim.
\end{proof}

Note that \(\omega = i\dd \left(|z|^{2\beta_1} + |z|^{2\gamma}\phi(z^{-1}w)\right)\) on \(\{R<r^{\alpha_0}\}\). In this region \(\rho \sim r\), and the proof of Lemma \ref{lem:COLLREG} gives
\[|\omega - \omega_H|_{\omega_H} = O(\rho^{\alpha_0-1}) . \]
In particular, \(\omega\) is well approximated by \(\omega_{\C_{\beta_1}}+\rho^{2\gamma/\beta_1}\omega_F\) on the colliding region \(\{R<r^{\alpha_0}\}\). See Figure \ref{fig:three lines}.

\begin{figure}
	\centering
	\scalebox{0.8}{
	\begin{tikzpicture}
	
	\centerarc[blue, thin](4,0)(110:250:4);
	\centerarc[red, thin](2.5,0)(100:260:2.5);
	
	\centerarc[blue, thin](-4,0)(-70:70:4);
	\centerarc[red, thin](-2.5,0)(-80:80:2.5);
	
	\draw[blue, dashed] (-2, 3.5) to (2, 3.5);
	\draw[blue, dashed] (-1, 2.7) to (1, 2.7);
	\draw[blue, dashed] (-.3, 1.5) to (.3, 1.5);
	
	\draw[blue, dashed] (-2, -3.5) to (2, -3.5);
	\draw[blue, dashed] (-1, -2.7) to (1, -2.7);
	\draw[blue, dashed] (-.3, -1.5) to (.3, -1.5);
	
	\draw[red, dashed] (1.6, 2.2) to (1.6,-2.2);
	\draw[red, dashed] (1.1, 2) to (1.1,-2);
	\draw[red, dashed] (.6, 1.5) to (.6,-1.5);
	
	\draw[red, dashed] (-1.6, 2.2) to (-1.6,-2.2);
	\draw[red, dashed] (-1.1, 2) to (-1.1,-2);
	\draw[red, dashed] (-.6, 1.5) to (-.6,-1.5);

	\draw [thin](-2,0) to (4,0);
	\draw [thick](0,-2) to (0,4);
	\draw [thick](-2, -0.5) to (4, 1);
	\draw [thick](-2, 1) to (4, -2);

	\draw (0.5, 3.1) ellipse (.05cm and .15cm);
	\draw (0, 3.1) to (0.5, 3.25);
	\draw (0, 3.1) to (0.5, 2.95);
	
	\draw (3, 1.25) ellipse (.25cm and .05cm);
	\draw (3, .75) to (2.75, 1.25);
	\draw (3, .75) to (3.25, 1.25);
	
	\draw (2, -0.5) ellipse (0.35cm and 0.05cm);
	\draw (1.65,-0.5) to (2,-1);
	\draw (2,-1) to (2.35,-0.5);
	
	\draw  (4, 0.2) node (asd) [scale=.9] {\(z\)};
	\draw  (0.2, 4) node (asd) [scale=.9] {\(w\)};
	\draw  (.9, 3.1) node (asd) [scale=.7] {\(2\pi\beta_1\)};
	\draw  (3, 1.5) node (asd) [scale=.7] {\(2\pi\beta_2\)};
	\draw  (2, -0.25) node (asd) [scale=.7] {\(2\pi\beta_3\)};
	
	\draw[thin](6,-1) to (12,-1);
	\draw[thick](6,-1) to (6,3);
	
	\centerarc[red, thin](6,9)(270:305:10);
	\draw[red, dashed] (9, -1) to (9,-.6);
	\draw[red, dashed] (10, -1) to (10,-.3);
	\draw[red, dashed] (11, -1) to (11, .3);
	
	\centerarc[blue, thin](6,4)(270:340:5);
	\draw[blue, dashed] (6, 0) to (8.9,0);
	\draw[blue, dashed] (6, 1) to (10,1);
	\draw[blue, dashed] (6, 2) to (10.5,2);
	
	\centerarc[thick](6,19)(270:288:20);
	\centerarc[thick](6,25)(270:284:26);
	
	\draw  (12, -1.2) node (asd) [scale=.9] {\(r\)};
	\draw  (6.2, 3) node (asd) [scale=.9] {\(R\)};
	\draw  (12, .2) node (asd) [scale=.9] {\(L_2\)};
	\draw  (12.1, -.6) node (asd) [scale=.9] {\(L_3\)};
	\end{tikzpicture}
	}
	\caption{The approximate solution in complex and polar coordinates. On the blue region \(\omega= \omega_{\cone}\), while \(\omega \sim \omega_{\C_{\beta_1}} + \rho^{2\gamma/\beta_1} \omega_F\) on the red part. The lines \(L_2\) and \(L_3\) approach  each other at super-linear rate \(\rho^{\gamma/\beta_1}\).} 
	\label{fig:three lines}
\end{figure}
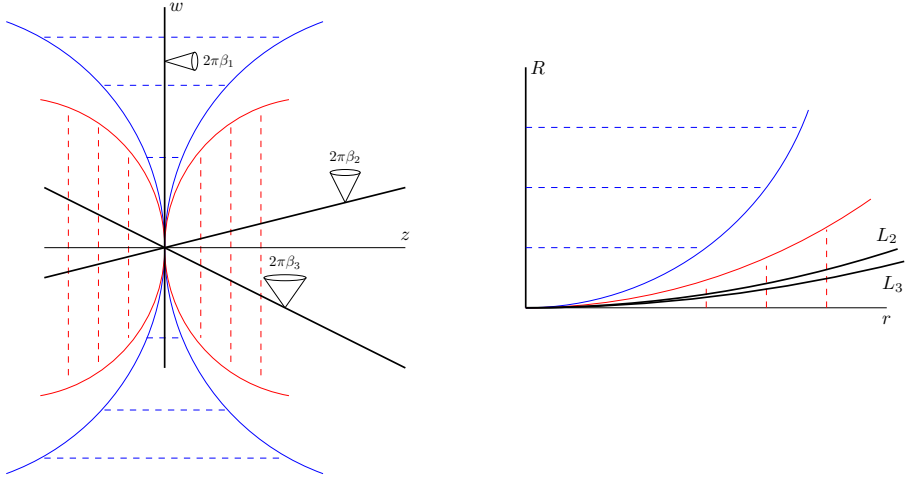

\begin{lemma} \label{lem:gluingreg}
	Recall that \(\beta_1/\gamma<\kappa=\alpha_0 \beta_1/\gamma<1\). On the gluing region \(|w| \sim |z|^{\kappa}\) we have
	\begin{equation*}
		\omega = \omega_{\cone} + E ,
	\end{equation*}
	with \(|E|_{\cone} = O(\rho^c)\) for some \(c>0\).
\end{lemma}

\begin{proof}
	Since \(|z^{-1}w| \gg 1\) we can use the expansion \eqref{fmpotential} for the potential of the flat metric to get
	\begin{equation*}
		\omega = i \dd \left( |z|^{2\beta_1} + |w|^{2\gamma} + \chi_2 |z|^{2\gamma} (A\log |z^{-1}w| + \phi_0) \right) .
	\end{equation*}
	Let \(f_1 = \chi_2 |z|^{2\gamma}\) and \(f_2 = A\log |z^{-1}w| + \phi_0\); so \(E = i \dd (f_1 f_2)\). It is straightforward to check that
	\begin{equation*}
		f_1 = O(|z|^{2\gamma}), \hspace{2mm} |\p f_1|_{\cone} = O(|z|^{2\gamma-\kappa \gamma}), \hspace{2mm} |\dd f_1|_{\cone} = O(|z|^{2\gamma-2\kappa\gamma})
	\end{equation*}
	and
	\begin{equation*}
	f_2 = O(-\log |z|), \hspace{2mm} |\p f_2|_{\cone} = O(|z|^{-\beta_1}), \hspace{2mm} |\dd f_2|_{\cone} = O(|z|^{-2\beta_1}) .
	\end{equation*}
	On the gluing region \(\rho \sim r=|z|^{\beta_1}\); it follows that \(|E|_{\cone} = O(\rho^c) \) for any \(0 < c < 2\gamma/\beta_1 - 2\kappa \gamma/\beta_1\).
\end{proof}

\begin{lemma}
	\(\omega\) is a smooth K\"ahler metric on the complement of the three lines.
\end{lemma}

\begin{proof}
	On the region \(\{R>2r^{\alpha_0}\}\) we have \(\omega = \omega_{\cone}>0\). On the gluing region \( \{ (1/3)R^{\alpha_0} < r < 3R^{\alpha_0} \} \), it follows from Lemma \ref{lem:gluingreg} that \(\omega>0\). Finally, on \(\{R<r^{\alpha_0}\}\), \(\omega>0\) thanks to Lemma \ref{lem:COLLREG}.
\end{proof}

In order to show that the approximate solution has standard conical singularities along \(L_2 \setminus \{0\}\) and \(L_3 \setminus \{0\}\) we will make use of the following general result

\begin{lemma} \label{localCSinglemma}
	Let \(0<\beta<1\) and let \(\omega\) be a smooth K\"ahler metric on \(B\setminus \{z_1=0\}\) with \(B \subset \C^2\) a ball around the origin. Assume that
	\[ \omega = \eta + i \dd (F|z_1|^{2\beta}) ,\]
	where \(\eta\) is a smooth \((1,1)\)-form such that $\eta(\partial/\partial z_2, \partial/\partial \bar{z}_2) >0$ along \(\{z_1 =0\}\) and $F$ is a smooth positive function. Then $\omega$ has standard cone singularities of angle \(2\pi\beta\) along \(\{z_1=0\}\) in a neighbourhood of the origin. 
\end{lemma}

\begin{proof}
	We compute
	\[ i \dd (F |z_1|^{2\beta}) = |z_1|^{2\beta} i \dd F + \beta |z_1|^{2\beta-2} \left( \bar{z}_1 idz_1 \bar{\partial}F + z_1 \partial F d\bar{z}_1 \right) + \beta^2 F |z_1|^{2\beta -2} idz_1 d\bar{z}_1 . \]
	Then at a point $p \in \{ z_1 = 0\}$, set $\tilde{z_1} = a z_1$, $\tilde{z_2} = b z_2$ with $a = F(p)^{1/2}$ and $b=\left( \eta(\partial/\partial z_2, \partial/\partial \overline{z}_2) (p) \right)^{1/2}$ to get \( \omega = \omega_{\C_{\beta}\times \C} + \sigma \) with \(\sigma\) a \(C^{\alpha}\) \((1, 1)\)-form, with exponent \(\alpha=(1/\beta)-1\),  vanishing at \(0\). 
\end{proof}

In the above proof we mean that a \((1,1)\)-form is \(C^{\alpha}\) if its components with respect to the co-frame \(\{v_i \wedge \bar{v}_j\}\) where \(v_1=|z_1|^{\beta_1-1}dz_1\) and \(v_2 = dz_2\) are H\"older continuous functions, see \cite{Donaldson}.

\begin{lemma}
	The approximate solution \(\omega\) has standard cone singularities of angle \(2\pi\beta_i\) along \(L_i \setminus \{0\}\) for \(i=1,2,3\).
\end{lemma}

\begin{proof}
	The statement is clear for \(L_1\), indeed \(\omega\) is isometric to \(\cone\) in a neighbourhood of \(L_1 \setminus \{0\}\).
	
	Consider the line \(L_2\). We use equation \eqref{locregFm} for the potential \(\phi\) in a neighbourhood of \(a_2\). Let \(p=(z_0, w_0) \in L_2\) with \(p \neq (0,0)\). In a neighbourhood of \(p\) we can write
	\begin{align*}
	\psi &= |z|^{2\beta_1} + |z|^{2\gamma}\phi(z^{-1}w) \\
	&= |z|^{2\beta_1} + e^{S_1} |z|^{2\gamma-2\beta_2} |w-a_2z|^{2\beta_2} + S_2 |z|^{2\gamma} .
	\end{align*}
	Note that \(|z|\) is a smooth non-vanishing function in a neighbourhood of \(p\). We change coordinates to \(z_1=w-a_2z\) and \(z_2=z-z_0\), so \(L_2 = \{z_1=0\}\). Our potential is
	\[\psi = F |z_1|^{2\beta_2} + \tilde{S}_2 , \]
	where \(F = e^{S_1} |z_2 +z_0|^{2\gamma-2\beta_2}\) and \(\tilde{S}_2 = |z_2 +z_0|^{2\beta_1} + S_2 |z_2 +z_0|^{2\gamma}\) are smooth in a neighbourhood of \(p =(0,0)\) because \(z_0 \neq 0\).
	We apply Lemma \ref{localCSinglemma} with \(\eta = i \dd \tilde{S_2}\). Note that \(\psi|_{L_2} = |z|^{2\beta_1} + \phi(a_2) |z|^{2\gamma}\), so \(\eta|_{L_2} >0 \). We conclude that \(\omega\) has standard cone singularities of angle \(2\pi\beta_2\) along \(L_2 \setminus \{0\}\), and similarly for \(L_3\).
\end{proof}

We move on to identify the tangent cone of the approximate solution at the origin. We introduce first the dilations of the cone \(\cone\).
For \(\lambda > 0\) we let
\begin{align*}
	D_{\lambda}(z,w) &= \lambda \cdot (z,w) \\ 
	&= (\lambda^{1/\beta_1}z, \lambda^{1/\gamma}w) .
\end{align*}
These satisfy the standard properties \(D_{\lambda}^* \omega_{\cone} = \lambda^2 \omega_{\cone}\) and \(\rho \circ D_{\lambda} = \lambda \rho\).
For \(t>0\), let \(B_t = \{ \rho < t \} \) and write \(B=B_1\).  Consider dilations maps \(D_{\lambda} : B \to B_{\lambda}\) for \(0< \lambda \ll 1\). Set
\begin{equation*}
	\omega_{\lambda} = \lambda^{-2} D_{\lambda}^* \omega .
\end{equation*}
Thus, we have a one parameter family of metrics \(\omega_{\lambda}\) on \(B\) with associated distance functions \(d_{\lambda}: B \times B \to \R_{\geq 0} \). We will show that \(d_{\lambda}\) converges uniformly to \(d_{\cone}\) as \(\lambda \to 0\), but first we start with an elementary auxiliary result.

\begin{lemma} \label{lem:FMTGCONE}
	Let \(\mu\) be a complex number with \(0<|\mu|<1\).
	The distance function of \(\omega_{F, \mu} = \gamma^2 |w- \mu a_2|^{2\beta_2-2} |w- \mu a_3|^{2\beta_3-2} idwd\bar{w}\) converges uniformly on compact sets to the distance function of \(\gamma^2 |w|^{2\gamma-2}idwd\bar{w}\) (i.e. \(\C_{\gamma}\)) as \(\mu \to 0\). More precisely, for every compact \(K \subset \C\) there is \(C>0\) such that
	\[|d_{\omega_{F, \mu}}(p,q) - d_{\C_{\gamma}}(p,q)| \leq C |\mu|^{\gamma} \]
	for all \(p, q \in K\).
\end{lemma}

\begin{proof}
	Introduce the multiplication map \(m_{\mu}(w) = \mu^{-1}w\), so
	\[\omega_{F, \mu} = |\mu|^{2\gamma} m_{\mu^{-1}}^* \omega_F . \]
	Up to a constant factor, the distance between \(\mu a_2\) and \(\mu a_3\) with respect to \(\omega_{F, \mu}\) is \(|\mu|^{\gamma}\). The three points \(0, \mu a_2, \mu a_3\) are contained in a small disc \(D_\mu\) around the origin, with diameter (with respect to either \(\omega_{F, \mu}\) or \(\C_{\gamma}\)) bounded by a constant times \(|\mu|^{\gamma}\). On the complement of \(D_{\mu}\), the forms \(\omega_{F, \mu}\) converge smoothly and uniformly to \(\omega_{\C_{\gamma}}\) as \(\mu \to 0\); hence the lemma follows.
\end{proof}

\begin{lemma} \label{lem:TGCONE}
	The tangent cone of \(\omega\) at the origin is \(\cone\). More precisely; the distance functions \(d_{\lambda}\) converge uniformly over \(B\) to \(d_{\cone}\)	as \(\lambda \to 0\). That is, for any \(\epsilon>0\) there is \(\lambda_0>0\) such that
	\begin{equation*}
	|d_{\lambda}(x_1, x_2) - d_{\cone}(x_1, x_2)| < \epsilon
	\end{equation*}
	for any \(x_1, x_2 \in B\) and \(0< \lambda < \lambda_0\).
\end{lemma}

\begin{proof}
Consider first the singular Hermitian metric
\begin{equation*}
	\omega_H = \beta_1^2 |z|^{2\beta_1-2} i dz d\bar{z} + \gamma^2 |w-a_2z|^{2\beta_2-2} |w-a_3z|^{2\beta_3-2} idwd\bar{w} .
\end{equation*}
	It is easy to see that, on the region \(|w|>(1/3) |z|^{\kappa}\), we have 
	\[
	    |\omega_H - \omega_{\cone}|_{\cone} = O(\rho^{c'})
	\]
	for some \(c'>0\). It follows from Lemma \ref{lem:COLLREG} that \(|\omega-\omega_H|_{\omega} = O(\rho^{\alpha_0-1})\) on the region \(|w|<|z|^{\kappa}\) (equivalently \(\{R<r^{\alpha_0}\}\)). By Lemma \ref{lem:gluingreg} \(\omega\) is also polynomially asymptotic to \(\omega_{\cone}\) on \(|w|>(1/3)|z|^{\kappa}\), we conclude that there is some \(c>0\) such that
	\begin{equation} \label{eq:COMPHERM}
		|\omega - \omega_H|_{\omega} = O(\rho^{2c})
	\end{equation}
	on all \(B\). We let \(\omega_{H, \lambda} = \lambda^{-2} D_{\lambda}^* \omega_H\) and write \(d_{H, \lambda}\) for the associated distance. It follows from equation \eqref{eq:COMPHERM} that
	\begin{equation} \label{eq:COMPHERM2}
		\| d_{\lambda} - d_{H, \lambda} \|_{L^{\infty}(B \times B)} = O (\lambda^c)
	\end{equation}
	as \(\lambda \to 0\).
	
	On the other hand, we can easily compute
	\begin{equation*}
		\omega_{H, \lambda} = \beta_1^2 |z|^{2\beta_1-2} i dz d\bar{z} + \gamma^2 |w- \lambda^{1/\beta_1-1/\gamma} a_2z|^{2\beta_2-2} |w-\lambda^{1/\beta_1 - 1/\gamma}a_3z|^{2\beta_3-2} idwd\bar{w} .
	\end{equation*}
	We use Lemma \ref{lem:FMTGCONE} with \(\mu = \lambda^{1/\beta_1 -1/\gamma}z \) to get
	\begin{equation} \label{eq:TGCONE}
		\| d_{H, \lambda} - d_{\cone} \|_{L^{\infty}(B \times B)} = O (\lambda^{\gamma/\beta_1 -1}) .
	\end{equation}
	Equations \eqref{eq:COMPHERM2} and \eqref{eq:TGCONE} together, imply that
	\begin{equation*}
		\| d_{\lambda} - d_{\cone} \|_{L^{\infty}(B \times B)} \to 0 \hspace{2mm} \mbox{as } \lambda \to 0 . \qedhere
	\end{equation*}
\end{proof}

\subsection{Ricci potential}
	
Let \(\Omega\) be the multi-valued holomorphic volume form  locally given by
	\begin{equation*}
	\Omega = (\beta_1 \gamma) \left(\prod_{j=1}^{3} \ell_j^{\beta_j-1}\right) dz dw .
	\end{equation*}
The Ricci potential \(h\) of the approximate solution is defined by  
\begin{equation*}
\omega^2 = e^{-h} \Omega \wedge \bar{\Omega} .
\end{equation*}
Recall that in the approximate solution metric ansatz we fixed some  \(1<\alpha_0<\gamma/\beta_1\).

\begin{proposition}\label{propriccipot}
	There is \(\delta > 2 \gamma / \beta_1\) such that in a neighbourhood of $0$
	\begin{equation*}
	|h| \leq \begin{cases}
	C \rho^{\delta -2} &\text{if } R>\mu \rho \\
	C \rho^{2 \gamma/\beta_1 - 2 \alpha_0}(-\log \rho) &\text{if } R \in ((\mu^{-1}/2) \rho^{\gamma/\beta_1}, 2\mu \rho) \\
	C \rho^{2\gamma/\beta_1 -2} &\text{if } R < \mu^{-1} \rho^{\gamma/\beta_1}
	\end{cases}
	\end{equation*}
	for suitable \(0 < \mu < 1\) and \(C>0\). In particular, \(|h| \leq C \rho^{\epsilon} \) for any \(0 < \epsilon < 2 \gamma/\beta_1 -2\alpha_0 \).
\end{proposition}

\begin{proof}
	We divide into five regions as follows.
	\newline
	
	\noindent
	\(\mathbf{I} = \{R > \mu r\}\). Here \(\rho\) is uniformly equivalent to \(R\). We have \(\chi_1 \equiv 1\) and \(\omega = \omega_{\cone}\).
	\begin{align*}
	e^{-h} &= |w|^{2\gamma-2} \left(|z-a_2w|^{2\beta_2-2} |z-a_3w|^{2\beta_3-2}\right)^{-1} \\
	&= |a_2 - z/w|^{2-2\beta_2} |a_3 - z/w|^{2-2\beta_3} \\
	&= 1 + O(|z/w|^2) .
	\end{align*}
	Note that have freedom, by performing a linear change of coordinates, to multiply all \(a_j\) by a non-zero constant  and keeping the weighted centre of mass at zero. Above we have used this freedom to set  \(1 = |a_2|^{2-2\beta_2}|a_3|^{2-2\beta_3}\). On this region we have \(|z| < C |w|^{\gamma/\beta_1}\), which implies \(|z/w| < C R^{1/\beta_1 -1/\gamma}\). We take
	\[ 2 \frac{\gamma}{\beta_1} < \delta < 2 + \frac{2}{\beta_1} - \frac{2}{\gamma} . \]
	\newline
	
	\noindent
	\(\mathbf{II} = \{2 r^{\alpha_0} < R < 2 \mu r \}\). In this region \(\rho\) is uniformly equivalent to \(r\). We have \(\chi_1 \equiv 1\) and \(\omega = \omega_{\cone}\). Same as before,
	\begin{equation*}
	e^{-h} = 1 + O(|z/w|^2) .
	\end{equation*}
	We still have \(|z/w| < C |w|^{1/\kappa -1} \ll 1\) and
	\begin{align*}
	|z||w|^{-1} &= (r^{1/\beta_1} R^{-1/\gamma} R) R^{-1} \\
	& \leq r^{1/\beta_1 - \alpha_0 (1/\gamma-1)}R^{-1} .
	\end{align*}
	We take
	\begin{equation*}
	\frac{\gamma}{\beta_1} < \frac{\delta}{2} < \frac{1}{\beta_1} - \left(\frac{1}{\gamma} -1\right) \alpha_0 .
	\end{equation*}
	\newline
	
	\noindent
	\(\mathbf{III} = \{r^{\alpha_0}/2 < R < 4 r^{\alpha_0} \}\).  In this region we still have \(|z^{-1}w| \gg 1\) and \(\rho\) is uniformly equivalent to \(r\), so
	\begin{equation*}
	\rho \sim r = |z|^{\beta_1} \sim |w|^{\beta_1/\kappa} = R^{1/\alpha_0} .
	\end{equation*}
	
	We use the expansion \eqref{fmpotential} to write the potential of \(\omega = i \dd u\) in the following form
	\begin{equation*}
	u = |z|^{2\beta_1} + |w|^{2\gamma} + \chi_2 |z|^{2\gamma} \left(A \log |z^{-1}w| + \phi_0(z^{-1}w) \right) .
	\end{equation*}
	We have
	\begin{equation*}
	\omega = \omega_{\cone} + E .
	\end{equation*}
	We want to estimate \(|E|_{\cone}\). We first consider
	
	\begin{equation*}
	i \dd (\chi_2 |z|^{2\gamma}) = \chi_2  i \dd (|z|^{2\gamma}) + |z|^{2\gamma} i \dd  \chi_2   + 2 \langle \p \chi_2, \p (|z|^{2\gamma} ) \rangle .
	\end{equation*}
	Note that
	\begin{equation*}
	| \p \chi_2 |_{\cone} = O(\rho^{-\alpha_0}), \hspace{2mm} 	| \dd \chi_2 |_{\cone} = O(\rho^{-2\alpha_0}) .
	\end{equation*}
	We get that
	\begin{align*}
	|E|_{\cone} &= O (\rho^{2\gamma /\beta_1 - 2\alpha_0} (-\log \rho)) \\
	&= O (\rho^{2\gamma /\beta_1} R^{-2} (-\log \rho)) .
	\end{align*}
	\newline
	
	\noindent
	\(\mathbf{IV} = \{\mu^{-1}r^{\gamma/\beta_1} < R < r^{\alpha_0}\}\). Here \(\rho \sim r\), \(\chi_2 \equiv 1\) and \(|z^{-1}w|> \mu^{-1} \gg 1\).  We use Equation \eqref{fmpotential} to write
	\begin{align*}
	\omega &= i \dd \left( |z|^{2\beta_1} + |w|^{2\gamma} + A |z|^{2\gamma} \log |z^{-1}w| + |z|^{2\gamma} \phi_0 (z^{-1}w) \right) \\
	&= \omega_{\cone} + E .
	\end{align*}
	We estimate the error term as follows
	\begin{align*}
	|E|_{\cone} &= O(|z|^{2\gamma -2 \beta_1} \log |z|^{-1}) \\
	&= O(\rho^{2\gamma/\beta_1-2} (-\log \rho)) \\
	&\leq C \rho^{2\gamma/\beta_1} R^{-2/\alpha_0} (-\log \rho) \\
	&= C \rho^{2\gamma/\beta_1} R^{2-2/\alpha_0} R^{-2} (-\log \rho) \\
	&\leq C \rho^{2\gamma/\beta_1 + \epsilon} R^{-2} (-\log \rho)
	\end{align*}
	with \(\epsilon = (\gamma/\beta_1)(2-2/\alpha_0) >0\). We take \(2\gamma / \beta_1 < \delta < 2\gamma/\beta_1 + \epsilon\).
	\newline
	
	\noindent
	\(\mathbf{V} =\{R < 2\mu^{-1}r^{\gamma/\beta_1} \}\). Here \(\rho \sim r\), \(\chi_2 \equiv 1\) and \(|z^{-1}w| \leq 2 \mu^{-1}\). Let \(\xi=z^{-1}w\) and compute the coefficients of
	\begin{equation*}
	\omega = i \dd \left( |z|^{2\beta_1} + |z|^{2\gamma} \phi (z^{-1}w) \right) .
	\end{equation*}
	as follows
	\begin{align*}
	g_{1\bar{1}} &= \beta_1^2 |z|^{2\beta_1-2} + \frac{\p^2}{\p z \p \bar{z}} \left(|z|^{2\gamma} \phi \right) \\
	&= \beta_1^2 |z|^{2\beta_1 -2} + |z|^{2\gamma-2} E_1 \\
	g_{1\bar{2}} &= |z|^{2\gamma-2} \frac{\p \phi}{\p \bar{\xi}} - |z|^{2\gamma-2} \xi \frac{\p^2 \phi}{\p \xi \p \bar{\xi}} \\
	&= |z|^{2\gamma-2} E_2 \\
	g_{2\bar{2}} &= \gamma^2 |w - a_2 z|^{2\beta_2 -2} |w - a_3 z|^{2\beta_3 -2} ,
	\end{align*}
	with error terms
	\begin{align*}
	E_1 &= \gamma^2 \phi - \gamma \xi \frac{\p \phi}{\p \xi} - \gamma  \bar{\xi} \frac{\p \phi}{\p \bar{\xi}} + |\xi|^2 \frac{\p^2 \phi}{\p \xi \p \bar{\xi}} , \\
	E_2 &= \gamma \frac{\p \phi}{\p \bar{\xi}} -  \xi \frac{\p^2 \phi}{\p \xi \p \bar{\xi}} .
	\end{align*}
	The volume form is
	\begin{equation*}
	\det (\omega) = \beta_1^2 \gamma^2 |\ell_1|^{2\beta_1 -2} |\ell_2|^{2\beta_2 -2} |\ell_3|^{2\beta_3 -2}  \left( 1 + E \right) 
	\end{equation*}
	with error 
	\begin{align*}
	E &= |z|^{2\gamma-2\beta_1} \left(E_1 - |z|^{2\gamma-2}|\ell_2|^{2-2\beta_2} |\ell_3|^{2-2\beta_3} |E_2|^2 \right) \\
	&= |z|^{2\gamma-2\beta_1} \left(E_1 - \left(\frac{\p^2 \phi}{\p \xi \p \bar{\xi}}\right)^{-1} |E_2|^2 \right) \\
	&= \gamma^2|z|^{2\gamma - 2\beta_1} \left( \phi - \left| \frac{\p \phi}{\p \xi} \right|^2 \left(\frac{\p^2 \phi}{\p \xi \p \bar{\xi}}\right)^{-1} \right) \\ &= O(\rho^{2\gamma / \beta_1 -2}) . \qedhere
	\end{align*} 
\end{proof}

In what follows we will only need that the Ricci potential is \(C^{\alpha}\). However, we have decided to include the more refined estimate in Proposition \ref{propriccipot} and the analysis into five regions drawing the analogy with \cite{sz}.

\subsection{Comparison with model cones}
In this section we compare geodesic balls in the conical line space $(\C^2,\omega)$ with balls centred at the apex of suitable `model cones.' There is a finite set \(\mathcal{C}\) consisting of six model cones:
\[\C^2, \hspace{1mm} \cone , \]
\[\C_{\beta_1} \times \C, \hspace{1mm} \C_{\beta_2} \times \C, \hspace{1mm} \C_{\beta_3} \times \C, \text{ and }\hspace{1mm}\C_{\gamma} \times \C; \]
such that for most scales balls in $(\C^2,\omega)$ look like balls centred at the apex of one of the model cones up to a fixed error.

To begin with, we consider two preliminary examples: First, the product of two cones \(\cone\), and second, the space with cones along two parallel lines \(\C \times \C_{F}\). Here, \(\C_F\) denotes the complex numbers endowed with the flat metric \(\omega_F\) with two cone points defined in Section \ref{sect:FMETRIC}.

Fix \(0<\lambda<1\). We consider geodesic balls $B(p,\lambda^k)$ in either space, and rescale the balls to unit size, \(\lambda^{-k}B(p, \lambda^{k})\). Here, we choose \(k \in \mathbf{Z}\), so that \(k \ll 0\) corresponds to large scales and as \(k\) increases we view smaller and smaller scales. Fix some \(\epsilon>0\) that quantifies the deviation from the model cones.

\begin{itemize}
	\item \(\cone\). Take \(0<\mu<1\) and divide into three regions: \(\mR_{\beta_1} = \{R>\mu^{-1}r\}\), \(\mR_{\beta_1, \gamma} = \{(\mu/2) r < R < (2/\mu) r\}\) and \(\mR_{\gamma} = \{R < \mu r\}\).
	After fixing a point $p \in \cone$, we write $R=R(p)$, $r=r(p)$, and $\rho = \rho(p)$.
	
	Now, at large scales, in the sense that
	\[\lambda^{-k} \rho < \epsilon,\]
	the balls \(\lambda^{-k}B(p, \lambda^{k})\) are isometric to a unit ball \(B(\tilde{p}, 1) \subset \cone\) with \(d(\tilde{p}, o) < \epsilon\), where \(o\) denotes the vertex (or origin) of \(\cone\). On the other hand, if
	\(\min \{r, R\} >0\) (so the point \(p\) lies on the complement of the two singular lines) and
	 \(k \gg 0\) is sufficiently large so that 
	\[\lambda^{k} < C \min \{r, R\}\] 
	with \(0<C=C(\beta_1, \gamma) \leq 1\),
	then  \(\lambda^{-k}B(p, \lambda^{k})\) is isometric to a Euclidean unit ball \(B_1 \subset \C^2\). 
	
	We are left to analyse the range of scales
	\[  C \min \{r, R\}  \leq \lambda^{k} \leq \epsilon^{-1} \rho  .  \]
	
	If \( p \in \mR_{\beta_1, \gamma}\) then we have \(\rho \sim r \sim R\), and the above range of \(k\) is uniformly bounded by \(\log(\epsilon^{-1})\), up to additive and multiplicative constants. 
	
	If \( p \in \mR_{\beta_1}\) the range is equivalent to \(r \leq \lambda^{k} \leq \epsilon^{-1}R\), up to constants. For those scales with \(\epsilon^{-1} r <\lambda^k < R\), the rescaled ball \(\lambda^{-k}B(p, \lambda^{k})\) is isometric to a unit ball in \(\C_{\beta_1} \times \C\) centred with distance \(< \epsilon\) from the apex (any of the points on the conical line can be considered as the apex). These cone models reduce the range of bad scales to
	\[ \{r \leq \lambda^{k} \leq \epsilon^{-1}R\} \cap \left( \{\lambda^k \leq \epsilon^{-1}r\} \cup
	 \{\lambda^{k} \geq R\} \right) \subset I_1 \cup I_2 .  \]
	Where \(I_1 = \{k: \,\ R \leq \lambda^{k} \leq \epsilon^{-1} R \}\) and \(I_2 = \{k: \,\ r \leq \lambda^{k} \leq \epsilon^{-1} r \}\), have length uniformly bounded by \(\log(\epsilon^{-1})\). See Figure \ref{fig:scales}. For the region \(\mR_{\gamma}\) we argue in the same way as for  \(\mR_{\beta_1}\).

	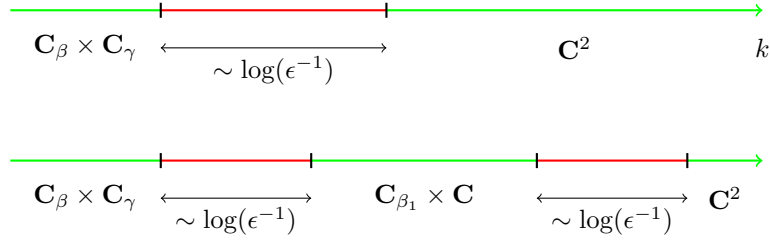
\begin{figure}
		\centering
		\begin{tikzpicture}
		
		\draw  (-.5, -.8) node (asd) [scale=1] {\(\sim \log (\epsilon^{-1})\)};
		\draw  (-3, -.5) node (asd) [scale=1] {\(\mathbf{C}_{\beta} \times \mathbf{C}_{\gamma}\)};
		\draw  (3.5, -.5) node (asd) [scale=1] {\( \mathbf{C}^{2}\)};
		\draw  (6, -.5) node (asd) [scale=1] {\(k\)};
		
		\draw[thick, green](-4,0) to (-2,0);
		\draw[thick, red](-2,0) to (1,0);
		\draw[thick, green, ->](1,0) to (6,0);
		\draw[thick](-2,-.1) to (-2,.1);
		\draw[thick](1,-.1) to (1,.1);
		
		\draw[<->](-2, -.5) to (1,-.5);

		\draw  (-1, -2.8) node (asd) [scale=1] {\(\sim \log (\epsilon^{-1})\)};
		\draw  (4, -2.8) node (asd) [scale=1] {\(\sim \log (\epsilon^{-1})\)};
		\draw  (-3, -2.5) node (asd) [scale=1] {\(\mathbf{C}_{\beta} \times \mathbf{C}_{\gamma}\)};
		\draw  (1.5, -2.5) node (asd) [scale=1] {\( \mathbf{C}_{\beta_1} \times \C \)};
		\draw  (5.5, -2.5) node (asd) [scale=1] {\(\C^2\)};
	
		\draw[thick, green](-4,-2) to (-2,-2);
		\draw[thick, red](-2,-2) to (0,-2);
		\draw[thick, green](0,-2) to (3,-2);
		\draw[thick, red](3,-2) to (5,-2);
		\draw[thick, green, ->](5,-2) to (6,-2);
		\draw[thick](-2,-2.1) to (-2,-1.9);
		\draw[thick](0,-2.1) to (0,-1.9);
		\draw[thick](3,-2.1) to (3,-1.9);
		\draw[thick](5,-2.1) to (5,-1.9);

		\draw[<->](-2, -2.5) to (0,-2.5);
		\draw[<->](3, -2.5) to (5,-2.5);

		\end{tikzpicture}
		\caption{Bad scales (red), good scales (green) and model cones for points in \(\mR_{\beta_1, \gamma}\) (top) and in \(\mR_{\beta_1}\) (bottom).}
		\label{fig:scales}
	\end{figure}

	\item \(\C \times \C_{F}\). It suffices to analyse \(\C_F\), as the cone models for the two parallel lines space are given by taking products with \(\C\). Fix \(0<\mu<1\) such that \(B(a_2, 3\mu) \cap B(a_3, 3\mu) = \emptyset\) and \(B(a_2, 3\mu), B(a_3, 3\mu)\) are isometric to balls centred at the apex of \(\C_{\beta_2}\) and \(\C_{\beta_3}\) respectively. We divide \(\C\) into three regions: \(\mR_{\beta_2} = \{d(w, a_2) < 2\mu\}\), \(\mR_{\beta_3} = \{d(w, a_3) < 2\mu\}\) and \(\mR_{\gamma} = \{d(w, a_2)> \mu, \,\ d(w, a_3)>\mu \} \). Let \(s = \min \{d(\cdot, a_2), d(\cdot, a_3) \} \) and let \(0<C=C(\beta_2, \beta_3)<1\) be such that \(B(p, Cs)\) is isometric to a Euclidean ball whenever \(p \in \C \setminus \{a_2, a_3\}\). W.l.o.g. we can assume that \(d(a_2, a_3)=1\). If \(k \ll 0\) is sufficiently negative so that
	\[\lambda^{-k}(s+1)< \epsilon, \]
	then \(\lambda^{-k}B(p, \lambda^k)\) has its two cone points at distance \(<\epsilon\) and its centre is at distance \(<\epsilon\) from the cone points; so we might say \(\lambda^{-k}B(p, \lambda^k)\) is \(\epsilon\)-close to the unit ball centred at the apex of \(\C_{\gamma}\). On the other hand, if \(k\) is sufficiently large so that
	\[\lambda^{k} < C s ; \]
	then \(\lambda^{-k}B(p, \lambda^k)\) is isometric to the Euclidean unit disc. We are left to analyse the range of scales
	\[ C s \leq \lambda^k \leq \epsilon^{-1}(s+1) . \]
	Up to constants, the length of this range is \(\log(\epsilon^{-1})+ \log(1+1/s)\). On \(\mR_{\gamma}\) the range is uniformly bounded because \(s\geq\mu\). On the region \(\mR_{\beta_2}\) we have \(s=d_2\) where \(d_2=d(a_2, \cdot)\). If
	\[ \epsilon^{-1} d_2 < \lambda^{k} < \mu  , \]
	then \(\lambda^{-k}B(p, \lambda^k)\) is isometric to a unit ball in \(\C_{\beta_2}\) with its centre at distance \(<\epsilon\) from the apex. Incorporating these model cones, the range of bad scales is reduced to 
	\[ \{C d_2 \leq \lambda^k \leq \epsilon^{-1}(d_2+1)\} \cap \left( \{\lambda^k \leq \epsilon^{-1}d_2\} \cup \{\lambda^k \geq \mu\} \right) \subset I_1 \cup I_2 . \]
	Here \(I_1=\{k: \,\ \mu \leq \lambda^k \leq \epsilon^{-1}(\mu+1) \}\) and \(I_2=\{k: \,\ Cd_2 \leq \lambda^k \leq \epsilon^{-1}d_2\}\). Up to constants, these two ranges of bad scales \(k\)'s have length uniformly bounded above by \(\log(\epsilon^{-1})\).
	
\end{itemize}

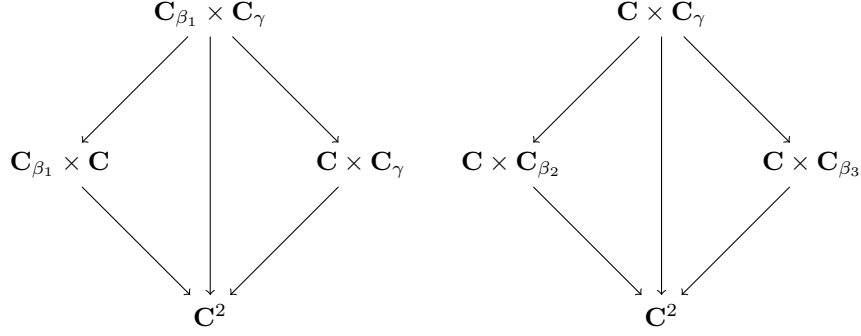
\begin{figure}
	\centering
	\begin{tikzpicture}  
	\node (a1) at (0,4) {\(\mathbf{C}_{\beta_1} \times \mathbf{C}_{\gamma}\)};  
	\node (a2) at (-2,2)  {\(\mathbf{C}_{\beta_1} \times \mathbf{C}\)};
	\node (a3) at (2,2)  {\(\mathbf{C} \times \mathbf{C}_{\gamma}\)};  
	\node (a4) at (0,0) {\(\mathbf{C}^2\)};  
	
	\draw[->] (a1) -- (a2); 
	\draw[->] (a1) -- (a3);  
	\draw[->] (a1) -- (a4);
	\draw[->] (a2) -- (a4);
	\draw[->] (a3) -- (a4);

	\node (a1) at (6,4) {\(\C \times \mathbf{C}_{\gamma}\)};  
	\node (a2) at (4,2)  {\(\C \times \mathbf{C}_{\beta_2}\)};
	\node (a3) at (8,2)  {\(\C \times \mathbf{C}_{\beta_3}\)};  
	\node (a4) at (6,0) {\(\mathbf{C}^2\)};

	\draw[->] (a1) -- (a2); 
	\draw[->] (a1) -- (a3);  
	\draw[->] (a1) -- (a4);
	\draw[->] (a2) -- (a4);
	\draw[->] (a3) -- (a4); 
	\end{tikzpicture}  
	\caption{Model cones for \(\cone\) (left) and \(\C \times \C_F \) (right).}
	\label{fig:MODELCONES}
\end{figure}

In Figure \ref{fig:MODELCONES} we summarize the above discussion about model cones in these two toy examples. Cones on the top model large scales and as we go along the arrows we pass to smaller scales. For each point there is a chain of model cones. The chain depends on which region the point lies, as follows.
\begin{itemize}
	\item Product of two cones: If \(p \in \mR_{\beta_1, \gamma}\) then the chain of model cones is \(\cone \to \C^2\); if \(p \in \mR_{\beta_1}\) then we have \(\cone \to \C_{\beta_1} \times \C \to \C^2\) (or it can also stop at \(\C_{\beta_1} \times \C\) if \(p\) lies on the line of cone angle \(2\pi\beta_1\)). Similarly for points in \(\mR_{\gamma}\).
	Since \(\cone\) is itself a cone, in the special case when \(p\) is the apex then \(\cone\) is always the model cone.
	
	\item Two parallel lines space: If \(p \in \mR_{\gamma}\) then the chain of model cones is \(\C \times \C_{\gamma} \to \C^2\); if \(p \in \mR_{\beta_2}\) then we have \(\C \times \C_{\gamma}  \to \C \times \C_{\beta_2} \to \C^2\) (or it can also stop at \(\C \times \C_{\beta_2}\) if \(p\) lies on the line of cone angle \(2\pi\beta_2\)). Similarly for points in \(\mR_{\beta_3}\). 
	
\end{itemize}

We say that two unit metric balls are \(\epsilon\)-close if their Gromov-Hausdorff distance is \(<\epsilon\). We extend this notion to balls of arbitrary radius by scale invariance, so \(B(p, r)\) and \(B'(p', r)\) are \(\epsilon\)-close if their Gromov-Hausdorff distance is \(<\epsilon r\). We now state our main result of this section.

\begin{proposition} \label{prop:MODELCONES}
	Let \(\epsilon>0\) and \(0<\lambda<1\). There is  \(N=N(\epsilon, \lambda)\) with the following property. If \(p\) is any point in \((\C^2, \omega)\) then for every \(k \geq 0\) except at most \(N\) of them, the balls \(\lambda^{-k}B(p, \lambda^k)\) are \(\epsilon\)-close to the unit ball centred at the apex of a model cone in \(\mathcal{C}\).
\end{proposition}

\begin{proof}
    If $p$ is the origin, then all sufficiently small balls are comparable to the tangent cone, which is the model cone $\mathbf C_{\beta_1} \times \mathbf C_{\gamma}$. 
    
	Otherwise, we divide into regions as follows.
	
	\(\mR_{\beta_1} = \{r < \mu R\}\). In this region \(\omega = \omega_{\cone}\), so the previous discussion applies. The chain of model cones for points in this region is \(\cone \to \C_{\beta_1} \times \C \to \C^2\) and the range of bad scales is contained, up to constants, in the union of the two intervals \(I_1 = \{k: \,\ R \leq \lambda^{k} \leq \epsilon^{-1} R \}\) and \(I_2 = \{k: \,\ r \leq \lambda^{k} \leq \epsilon^{-1} r \}\) which have length uniformly bounded above by \(\log(\epsilon^{-1})\).
	
	\(\mR_{\beta_1, \gamma} = \{(\mu/2) r <  R < (2/\mu) r \}\). Here \( r \sim R \sim \rho\) and we still have \(\omega = \omega_{\cone}\), so the previous discussion applies again. The chain of model cones for points in this region is \(\cone \to \C^2\) and the range of bad scales is contained, up to constants, in the interval \(I = \{k: \,\ \rho \leq \lambda^{k} \leq \epsilon^{-1} \rho \}\) which has length uniformly bounded above by \(\log(\epsilon^{-1})\).
	
	\(\mR_{\gamma}^{'} = \{R<\mu r\} \cap \{R > 2 r^{\alpha_0}\}\). Here \(\rho \sim r\) and \(\omega = \omega_{\cone}\). As in the case of \(\cone \) above, the chain of model cones is \(\cone \to \C \times \C_{\gamma} \to \C^2\) and the range of bad scales is contained, up to constants, in the union of the two intervals \(I_1 = \{k: \,\ r \leq \lambda^{k} \leq \epsilon^{-1} r \}\) and \(I_2 = \{k: \,\ R \leq \lambda^{k} \leq \epsilon^{-1} R \}\) which have length uniformly bounded above by \(\log(\epsilon^{-1})\).
	
	\(\mR_{\gamma}^{''} = \{R < 3 r^{\alpha_0}\} \cap \{s> \mu r^{\gamma/\beta_1}\} \),
	 where \(s=\min\{d_2, d_3\}\) and \(d_2, d_3\) are the distances to the conical lines \(L_2, L_3\).  Here \(\rho \sim r\) and  \(\omega= \omega_{\C_{\beta_1}} + r^{2\gamma/\beta_1} \omega_F + E \) up to a small error \(|E|_{\omega} = O(\rho^{-1}R)\). For \(\lambda^k > \epsilon^{-1} r\) the balls \(\lambda^{-k}B(p, \lambda^k)\) are \(\epsilon\)-close to \(B_1 \subset \cone\) as in Lemma \ref{lem:TGCONE}. For \(\lambda^k < r\) we introduce a cut in the \(\C_{\beta_1}\) factor and rescale balls centred at \(p\) by \(r^{-\gamma/\beta_1}\)
	to reduce to the situation \(\mR_{\gamma} \subset  \C \times \C_F\) considered before. If \((s+1)\lambda^{-k} < \epsilon r^{\gamma/\beta_1}\) then \(\lambda^{-k}B(p, \lambda^k)\) is \(\epsilon\)-close to \(B_1 \subset \C \times \C_{\gamma}\). If \(\lambda^k < \mu r^{-\gamma/\beta_1} s\) then \(\lambda^{-k}B(p, \lambda^k)\) is \(\epsilon\)-close to \(B_1 \subset \C^2\).
	
	\(\mR_{\beta_2}=\{R<3r^{\alpha_0}\}\cap\{d_2<2\mu r^{\gamma/\beta_1}\}\).  Here \(\rho \sim r\) and \(s=d_2\), same as before  \(\omega= \omega_{\C_{\beta_1}} + r^{2\gamma/\beta_1} \omega_F + E \) with \(|E|_{\omega} = O(\rho^{-1}R)\). For \(\lambda^k > \epsilon^{-1} r\) the balls \(\lambda^{-k}B(p, \lambda^k)\) are \(\epsilon\)-close to \(B_1 \subset \cone\) as in Lemma \ref{lem:TGCONE}. For \(\lambda^k < r\) we introduce a cut in the \(\C_{\beta_1}\) factor and rescale balls centred at \(p\) by \(r^{-\gamma/\beta_1}\)
	to reduce to the situation \(\mR_{\beta_2} \subset  \C \times \C_F\) considered before. If \((d_2+1)\lambda^{-k} < \epsilon r^{\gamma/\beta_1}\) then \(\lambda^{-k}B(p, \lambda^k)\) is \(\epsilon\)-close to \(B_1 \subset \C \times \C_{\gamma}\). If \(\lambda^k < \mu r^{-\gamma/\beta_1} d_2\) then \(\lambda^{-k}B(p, \lambda^k)\) is \(\epsilon\)-close to \(B_1 \subset \C^2\). We interpolate by introducing the range of good scales \(\lambda^{-k}d_2<\epsilon r^{\gamma/\beta_1}\) modelled by the cone \(\C \times \C_{\beta_2}\).
	
	\(\mR_{\beta_3}=\{R<3r^{\alpha_0}\}\cap\{d_3<2\mu r^{\gamma/\beta_1}\}\) is symmetric to \(\mR_{\beta_2}\).
	
	We are choosing \(\mu\) sufficiently small so that \(\{d_2<3\mu r^{\gamma/\beta_1}\} \cap \{d_3 < 3\mu r^{\gamma/\beta_1}\} = \emptyset \). In particular this implies that 
	\[\mR_{\beta_2} \cup \mR_{\beta_3} \cup\mR_{\gamma}^{''} \cup \mR_{\gamma}^{'} = \{R < \mu r\}  \]
	and the regions cover. 
\end{proof}

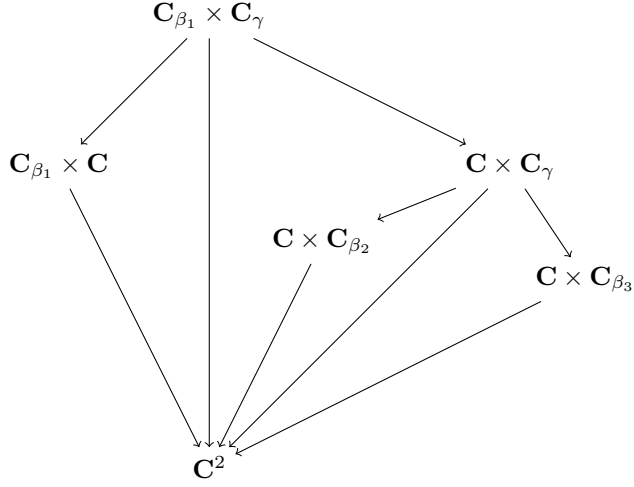
\begin{figure}
	\centering
	
	\begin{tikzpicture}
	
	\node (a1) at (0,6) {\(\mathbf{C}_{\beta_1} \times \mathbf{C}_{\gamma}\)};  
	\node (a2) at (-2,4)  {\(\mathbf{C}_{\beta_1} \times \mathbf{C}\)};
	\node (a3) at (4,4)  {\(\mathbf{C} \times \mathbf{C}_{\gamma}\)};
	\node (a4) at (1.5,3) {\(\mathbf{C} \times \mathbf{C}_{\beta_2}\)};
	\node (a5) at (5,2.5) {\(\mathbf{C} \times \mathbf{C}_{\beta_3}\)}; 
	\node (a6) at (0,0) {\(\mathbf{C}^2\)};  
	
	\draw[->] (a1) -- (a2); 
	\draw[->] (a1) -- (a3);  
	\draw[->] (a1) -- (a6);
	\draw[->] (a2) -- (a6);
	\draw[->] (a3) -- (a4);
	\draw[->] (a3) -- (a6);
	\draw[->] (a3) -- (a5); 
	\draw[->] (a4) -- (a6);
	\draw[->] (a5) -- (a6);         
	
	\end{tikzpicture}
	\caption{Chains of model cones for $(\C^2,\omega)$. Note that as we follow the arrows, passing to smaller scales, the volume densities (i.e. product of the cone angle factors) increase.}
	\label{fig:TREEMODELCONES}
\end{figure}

We summarize the proof of Proposition \ref{prop:MODELCONES} in Figure \ref{fig:TREEMODELCONES}, illustrating the chains of model cones.  The diagram is a coupling of the two diagrams in Figure \ref{fig:MODELCONES}.

\subsection{Quasi-isometry to \(\R^4\)}

\begin{lemma} \label{lem:QUASIISOM}
	Up to a (singular) change of coordinates, the approximate solution is quasi-isometric to the Euclidean metric. More precisely, there is a diffeomorphism \(\tilde{\Phi}\) on the complement of the three lines that extends over the lines as a continuous homeomoprhism and such that
	\begin{equation*}
	C^{-1}g_{\R^4} \leq \tilde{\Phi}^* g \leq C g_{\R^4} 
	\end{equation*}
	for some uniform \(C>0\).
\end{lemma}

\begin{proof}
	The starting point is that if we pull-back the cone metric \(g_{\beta}=\beta^2 |z|^{2\beta-2}|dz|^2\) by \(u \to z= |u|^{1/\beta-1}u\), then \(g_{\beta}\) is quasi-isometric to the Euclidean metric \(g_{\R^2}=|du|^2\). In the sense that \(C^{-1}g_{\R^2} \leq g_{\beta} \leq C g_{\R^2}\) for some uniform \(C>0\).
	
	Consider now \((\C, g_F)\) with 
	\[g_F = \gamma^2 |w-a_2|^{2\beta_2-2} |w-a_3|^{2\beta_3-2} |dw|^2 . \]
	Introduce polar coordinates \((r_2,\theta_2)\), \((r_3, \theta_3)\) around the cone points \(a_2\) and \(a_3\). Let \(\Phi: \C \to \C\) be given by
	\begin{equation*}
	\Phi(v) = \begin{cases}
	|v|^{\gamma^{-1}{\chi_1}(|v|) -1} v & \mbox{ if } |v| > \Lambda \\
	r_2^{\beta_2^{-1} \chi_2(r_2)} e^{i\theta_2} + a_2 & \mbox{ if } |v-a_2| < \mu \\
	r_3^{\beta_3^{-1} \chi_3(r_3)} e^{i\theta_3} + a_3 & \mbox{ if } |v-a_3| < \mu \\
	v & \mbox{ elsewhere. } 
	\end{cases}
	\end{equation*}
	Here we have fixed some large \(\Lambda\) and small \(\mu\) so that the first three regions are pairwise disjoint. We have used a standard cut-off functions:
	\begin{align*}
	\chi_1(t) = 1 \,\ \mbox{ if } t \geq 3\Lambda, \hspace{3mm} \chi_1(t) = \gamma \,\ \mbox{ if } t \leq 2\Lambda, \\
	\chi_2(t) = 1 \,\ \mbox{ if } t \leq \mu/3, \hspace{3mm} \chi_2(t) = \beta_2 \,\ \mbox{ if } t \geq \mu/2 \\
	\chi_3(t) = 1 \,\ \mbox{ if } t \leq \mu/3, \hspace{3mm} \chi_3(t) = \beta_3 \,\ \mbox{ if } t \geq \mu/2 .	
	\end{align*}
	This way \(\Phi\) is a diffeomorphism of \(\C \setminus \{a_2, a_3\}\) that extends as a homeomorphism of \(\C\) fixing the cone points \(a_2, a_3\). Moreover, if we write \(w=\Phi(v)\), then 
	\[C^{-1} |dv|^2 \leq \Phi^*(g_F) \leq C |dv|^2 , \]
	for some uniform \(C>0\).
	
	Consider now the Hermitian metric
	\begin{align*}
	g_H &= \beta_1^2 |z|^{2\beta_1-2} |dz|^2 + \gamma^2 |w-a_2z|^{2\beta_2-2}|w-a_3z|^{2\beta_3-2} |dw|^2 \\
	&= \beta_1^2 |z|^{2\beta_1-2} |dz|^2  + |z|^{2\gamma} m_{z^{-1}}^*(g_F) ,
	\end{align*}
	where \(m_{z^{-1}}(w)=z^{-1}w\). For \(z \neq 0\),  define 
	\[\Phi_z = m_{z} \circ \Phi \circ m_{|z|^{-\gamma+1}z^{-1}} . \]
	This way \(\Phi_z(v) = |v|^{1/\gamma-1} v \) when \(|v| \geq 3 \Lambda |z|^{\gamma}\); so \(\Phi_z\) converges uniformly to \(v \to |v|^{1/\gamma-1}v\) as \(z \to 0\). Moreover, 
	\[C^{-1} |dv|^2 \leq \Phi_z^* \left( |z|^{2\gamma} m_{z^{-1}}^*(g_F) \right) \leq C |dv|^2 . \]
	Finally, we set
	\begin{equation*}
	\tilde{\Phi} (u, v) = \left(|u|^{\beta_1-1} u, \Phi_{|u|^{\gamma-1}u} (v)  \right)
	\end{equation*}
	and conclude that
	\begin{equation*}
	C^{-1}g_{\R^4} \leq \tilde{\Phi}^*(g_H) \leq C g_{\R^4} .
	\end{equation*}
	Since \(g_H\) is uniformly equivalent to the approximate solution metric \(g\), the lemma is proved.
\end{proof}

\section{Schauder Estimate} \label{sect:SCHAUDER}

On our approximate solution \((\C^2, \omega)\), it is straightforward to set up existence of weak solutions for the Poisson equation \(-\Delta u = f\).
We then use subquadratic harmonic polynomials to establish H\"older continuity for the complex Hessian of weak solutions via approximation and integral estimates. The technique is standard in PDE, see \cite[Chapter 5.4]{HanLin}. For related applications of this technique, in the setting of complete Ricci flat manifolds with maximal volume growth, see \cite{Ch19} and \cite{sz2}.

\subsection{Weak solutions}
For a regular domain \(\Omega \subseteq (\C^2, \omega)\), we define \(W^{1,2}(\Omega)\) as the completion of the Lipschitz functions under the norm
\[\|u\|_{W^{1,2}(\Omega)} = \left(\int_{\Omega}u^2 + \int_{\Omega}|\nabla u|^2 \right)^{1/2} .\]
In coordinates \(\tilde{\Phi}\) of Lemma \ref{lem:QUASIISOM} the space \(W^{1,2}\) agrees with the usual one of \(\R^4\).
Given  \(f \in L^1_{\mathrm{loc}}(\Omega)\), we say that \(u \in W^{1,2}(\Omega)\) is a weak solution of \(-\Delta u = f\), if 
\begin{equation}\label{eq:WEAKSOL}
\int_{\Omega} \langle \nabla u, \nabla \psi \rangle =  \int_{\Omega} f \psi  
\end{equation}
for every Lipschitz test function \(\psi\) with compact support contained in \(\Omega\).

Some parts of the standard theory of weak solutions to the Poisson equation carry over in a straightforward manner to our conical line space. We state the relevant inequalities in a scale invariant way.
Let \(B_r=B(p, r)\) be a geodesic ball in $(\C^2,\omega)$. Throughout the paper we use the scale-invariant \(L^2\)-norms
\begin{equation}
	\|u\|_{B_r} = \left( r^{-4} \int_{B_r} u^2 \right)^{1/2} .
\end{equation}

\begin{lemma}
	We have the following
	\begin{enumerate}
		\item Ahlfors regularity.
		\begin{equation*}
		C^{-1}r^4 \leq \mathrm{vol}(B_r) \leq C r^4 .
		\end{equation*}
		\item Rellich compactness. The inclusion \[W^{1,2}(B_r) \subset L^2(B_r)\] is compact.
		\item Poincar\'e inequality. If either the average of \(u \in W^{1, 2}\) on  \(B_r =B(x, r)\) is zero, or if its compactly supported, then
		\begin{equation*}
		\|u\|_{B_r} \leq C r\|\nabla u\|_{B_r} .
		\end{equation*}
		\item Caccioppoli inequality. Let \(u \in W^{1, 2}(B_r)\) solve \(-\Delta u = f\) with \(f \in L^2(B_r)\), then
		\begin{equation*}
			r\|\nabla u\|_{B_{r/2}} \leq r^2\|f\|_{B_r} + C\|u\|_{B_r} .
		\end{equation*}
	\end{enumerate}
\end{lemma}

\begin{proof}
	The first three items follow from the fact that \(\omega\) is quasi-isometric to the Euclidean metric by Lemma \ref{lem:QUASIISOM}; indeed the Sobolev norms are uniformly equivalent to the standard Euclidean ones after pulling back functions by the map \(\Phi\) of Lemma \ref{lem:QUASIISOM}.
	The last item follows by testing Equation \eqref{eq:WEAKSOL} against \(\psi = \eta^2 u\), with \(\eta\) a compactly supported function in \(B_r\) equal to \(1\) on \(B_{r/2}\).
\end{proof}

A standard contradiction argument, that combines
the Caccioppoli inequality together with Rellich compactness, gives the following.

\begin{lemma}[Harmonic approximation] \label{lem:HARMAPROX}
		For every \(\epsilon>0\) there is \(\delta>0\) with the following property: If \(u \in W^{1, 2}(B_1)\) satisfies \(\Delta u = f\) with \(\|u\|_{L^2(B_1)} \leq 1\) and \(\|f\|_{L^2(B_1)} < \delta\), then there is a weak harmonic function \(h \in W^{1, 2}(B_{1/2})\) such that \(\|u-h\|_{L^2(B_{1/2})} < \epsilon\).
\end{lemma}

It follows from the Ahlfors regularity of the measure that Campanato's \(L^2\) characterization of H\"older spaces holds \cite{Ca63,Ca65,GilTru,HanLin}. We will measure H\"older continuity in terms of these \(L^2\) norms: 
\begin{equation}
	c^{-1} |u|_{\alpha} \leq r^{-\alpha}\Big \|u - \fint_{B_r} u \Big \|_{B_r} \leq |u|_{\alpha} .
\end{equation}

Another consequence of the quasi-isometry with \(\R^4\) is that we can apply De Giorgi-Nash-Moser and  conclude that weak harmonic functions are indeed H\"older continuous \cite{GilTru}. We will proceed to establish higher order estimates for these weak harmonic functions.

\subsection{Subquadratic harmonic polynomials and reference functions}

We begin by recalling the space of subquadratic homogeneous harmonic functions on \(\coneb\) with angles \(0<\beta \leq 1\) and \(0<\gamma\leq 1\) as in \cite[Section 3]{dBoEdw}.\footnote{Here we discuss the general picture, but later we will restrict to \(\beta=\beta_1, \beta_2, \beta_3, 1\) and \(\gamma=\gamma, 1\) depending on the model cone.} We regard \(\coneb\) as a cone whose link is a three-sphere endowed with a constant sectional curvature \(1\) metric with conical singularities along two Hopf circles (or only one if either \(\gamma=1\) or $\beta=1$, and none if \(\beta=\gamma=1\)). This singular metric on the three-sphere is quasi-isometric to the standard round metric. As a consequence, the Friedrich extension of the Laplacian
has discrete spectrum and the corresponding \(W^{1,2}\) eigenfunctions form an orthogonal basis of \(L^2\). For each eigenfunction \(-\Delta_{S^3}\psi=\lambda \psi\) on the three-sphere there are two  corresponding homogeneous harmonic functions \(u=\rho^{d_{\pm}} \psi\) on \(\coneb\) with \(d_{+} \geq 0\) and \(d_{-}<0\)  the two solutions of the indicial equation
\[d(d+2)=\lambda .\]
The set of all such \(d_{\pm}\) is called the indicial root set \(\mathcal{I} \subset \R\). It is a discrete set, symmetric with respect to \(-1\), and \(\mathcal{I}\cap(-2,0)=\emptyset\).

We call the functions with $0\leq d \leq 2$ the subquadratic harmonics and define \(\mathcal{H}_{\leq 2}\) to be the subspace of $L^2$ spanned by the homogeneous subquadratic harmonic functions.

\begin{proposition}\label{prop:HARMPOL}
	On the cone \(\coneb\) with coordinates \((z,w)\), the space \(\mH_{\leq 2}\) is spanned by:
	\begin{itemize}
		\item the constant \(1\);
		\item  the real and imaginary parts of \(z\) if \(\beta \geq 1/2\), and \(z^2\) if \(\beta=1\);
		\item  the real and imaginary parts of \(w\) if \(\gamma \geq 1/2\), and \(w^2\) if \(\gamma=1\);
		\item the real and imaginary parts of \(zw\) if \(\beta=\gamma=1\);
		\item the quadratic harmonic function \(\beta^{-2} |z|^{2\beta} - \gamma^{-2} |w|^{2\gamma}\).
	\end{itemize}
\end{proposition}

The above proposition is proved by separation of variables as we sketch below, see \cite[Proposition 3.4]{dBoEdw} for more details. 
A key property that we will exploit is that \(|\dd P|_{\alpha} = 0\) for every \(P \in \mH_{\leq 2}\).

\begin{proof}
	We assume that at least one cone factor is not Euclidean, say \(\beta<1\).
	The starting point is that homogeneous harmonic functions of \(\C_{\beta}\) are given by the real and imaginary parts of \(z^k\) with \(k\) an integer number. In particular, the growth rates are of the form \(k/\beta\).
	Write \(\Delta_{\beta}\) for the Laplace operator of the \(\C_{\beta}\) factor and similarly for \(\Delta_{\gamma}\). These operators lower the degree of homogeneous functions by two.
	Let \(f\) be a homogeneous harmonic of degree \(0 < d \leq 2\), then
	\begin{equation*}
	\Delta_\beta f + \Delta_{\gamma} f = 0 .
	\end{equation*}
	The functions \(\Delta_{\beta} f\) and \(\Delta_{\gamma} f\) are homogeneous harmonic of degree \(d-2 \in (-2, 0]\). It is a general fact that Riemannian cones of real dimension \(m\geq 4\) have no homogeneous harmonic functions whose degree belongs to the interval \((2-m, 0)\). This also holds for cones over stratified spaces \cite{ACM} and also for Ricci limit spaces, as follows from the separation of variables formula for the Laplacian of the cone \cite[Theorem 4.15]{ding} together with the fact that the Laplcian on the link of the cone is definite. We conclude\footnote{Here one should prove that \(\Delta_{\beta} f = \rho^{d-2} \varphi\) with \(\varphi\) a \(W^{1,2}\) function on the link.} that
	the functions \(\Delta_{\beta} f\) and \(\Delta_{\gamma} f\) have homogeneous degree equal to zero; i.e. \(d=2\). This implies that \(\Delta_{\beta} f \equiv c\) and \(\Delta_{\gamma} f \equiv -c\) for some constant \(c\). If \(c=0\) we can regard \(f\), via \(w \to f(\cdot, w)\), as a map from \(\C\) to the vector space of subquadratic harmonic polynomials of \(\C_{\beta}\). We get
	\begin{equation*}
	f(z, w) = f_0 (w) + f_1(w) z + f_{\bar{1}}(w) \bar{z} .
	\end{equation*}
	Since \(\Delta_{\gamma} f =0\), we must have that \(f_0, f_1, f_{\bar{1}}\) are subquadratic harmonic functions of \(\C_{\gamma}\). We conclude that \(f_0 = a + b w + c \bar{w} \) and that \(f_1, f_{\bar{1}}\) are constants.
	
	If \(\Delta_{\beta} f \equiv c \neq 0\), we can assume that \(c=4\). Up to adding  subquadratic harmonic functions, the only function of \(\C_{\beta}\) of subquadratic growth with \(\Delta_{\beta} u =4 \) is \(u = |z|^{2\beta}\). Similarly as before, we write
	\begin{equation*}
	f(z, w) = f_0 (w) + f_1(w) z + f_{\bar{1}}(w) \bar{z} + |z|^{2\beta} .
	\end{equation*}
	Since \(\Delta_{\gamma} f =-4\), we conclude that \(f_1, f_{\bar{1}}\) are constants and \(f_0 = a + b w + c \bar{w} + dw^2 + e\bar{w}^2 - |w|^{2\gamma} \).
\end{proof}


Next, we define spaces of reference functions for balls which are $\epsilon$-close to balls at the apex of a model cone, consisting of suitable approximations to the subquadratic harmonic functions.

\begin{proposition} \label{prop:REFFUNCT}
	For  \(0<\epsilon \ll 1\) and \(B(x, \rho) \subset (\C^2, \omega)\) geodesic ball of radius $\rho$, which is \(\epsilon\)-close to a ball at the apex of a model cone \(C(Y)\in \mathcal{C}\), we can define finite dimensional spaces \(\mH_{\leq 2}^{\epsilon}(B(x, \rho))\) of reference functions on \(B(x, \rho)\)  with the property that
	\(\mH_{\leq 2}^{\epsilon}(B(x, \rho))\) converges to \(\mH_{\leq 2}(C(Y))\) as subspaces in $L^2$ as \(\epsilon \to 0\). 
	Moreover, if \(P \in \mH_{\leq 2}^{\epsilon}\) then \(|\dd P|_{\alpha}\) is bounded.
\end{proposition}

\begin{proof}
	If \(B(x, \rho)\) is \(\epsilon\)-close to a ball in \(C(Y)\in \mathcal{C}\), then the Gromov-Hausdorff distance between \(\rho^{-1}B(x, \rho)\) and \(B_1 \subset C(Y)\) is \(< \epsilon\) as in Proposition \ref{prop:MODELCONES}. We can change coordinates (as in Section \ref{sect:APROXSOL}) and realize both \(\rho^{-2}\omega|_{B(x, \rho)} = i \dd \tilde{\psi}\) and \(\omega_{C(Y)}\) as being defined on \(\{|\tz|^{2\beta}+|\tw|^{2\gamma}<2\}\). We define a basis of \(\mH_{\leq 2}^{\epsilon}(B(x, \rho))\) given by the appropriate real and imaginary parts of \(\tilde z,\tilde w,\tilde w^2,\tilde z^2,\tilde z \tilde w\) according to those which belong to \(\mathcal{H}_{\leq 2}(C(Y))\),  and \(2|\tz|^{2\beta}-\tilde{\psi}\) which will converge uniformly to \(|\tz|^{2\beta}-|\tw|^{2\gamma}\). Indeed the potential \(\tilde{\psi}\) converges uniformly to \(|\tz|^{2\beta}+|\tw|^{2\gamma}\) as $\epsilon \to 0$, and the convergence of the subspace \(\mH_{\leq 2}^{\epsilon}(B(x,\rho))\) to \(\mH_{\leq 2}(C(Y))\) is then clear.
	
	We now prove the bound on \(|\dd P|_{\alpha}\). Note that, since \(\dd\) of the potential is constant (being equal to the metric itself), it is enough to bound \(\dd|z|^{2\beta}\) on each of the model cones. If the model cone contains a Euclidean \(\C\)-factor, then we can take \(\beta=1\). In this case \(|z|^2\) is smooth, its \(\dd\) is bounded with respect to any smooth metric and, since the coefficients of our approximate solution are uniformly bounded from below by the ones of the Euclidean metric, the bound on \(|\dd |z|^2|_{\alpha}\) follows. The only model cone which does not contain a Euclidean factor is \(\cone\). Here, we use the approximate Hermitian metric to bound the H\"older coefficient for \(|\dd |z|^{2\beta_1}|_{\alpha, \omega_H}\) and also for its scalings \(\omega_{H, \lambda}=\lambda^{-2}D_{\lambda}^*\omega_H\). Since \(\omega_{H} - \omega\) has polynomial asymptotic decay at the origin, the result follows.
\end{proof}

 \begin{remark}
 	The reference functions \(\mH_{\leq 2}^{\epsilon}(B(x, \rho))\) in the proof of Proposition \ref{prop:REFFUNCT} are \emph{not} necessarily harmonic. However, we can check that \(\Delta P (x) \to 0\) as \(\epsilon \to 0\). Indeed, it is only necessary to check this for \(|z|^{2\beta}-\psi\). Here we have $\Delta \psi \equiv 2$, and in the limit when \(\epsilon \to 0\) we can replace the Laplacian with that of the Hermitian metric $\omega_H$ for which \(\Delta |z|^{2\beta} \to 2\).
 	
    By subtracting small multiples of the potential,  replacing $P$ with $P - \frac 1 2 (\Delta P(x)) \tilde \psi$, we can assume that \(\Delta P(x) =0\) for all $P \in \mH_{\leq 2}^\epsilon( B(x,\rho))$. 
 \end{remark}
 
 \begin{remark}
	We could have defined the reference functions to be harmonic by taking harmonic approximations of each of the reference functions, but our proof would require that the complex Hessian of these approximations be bounded in \(C^{\alpha}\).
 \end{remark}

We briefly state the spectral decomposition lemma for model cones from \cite[Lemma 4.2]{dBoEdw}.
\begin{lemma}[Spectral decomposition]
	Let $0<\lambda<1$ and $d \geq 0$. For a cone $C(Y)$, let $d_*$ be the smallest indicial root greater than $d$.
	
	If $f$ is harmonic on $B(0,1) \subset C(Y)$ and $L^2$-orthogonal to $\mathcal H_{\leq d}(B(0,1))$, then
	\[
		\| f \|_{B(0,\lambda)} \leq \lambda^{d_*} \| f \|_{B(0,1)}
	\]
	with equality if and only if $f$ is homogeneous of degree $d_*$.
\end{lemma}

\begin{proof}
	The proof is contained in \cite{dBoEdw}, we give a sketch here. 
	
	Let $\{ \phi_i \}$ be an $L^2$-orthonormal basis of eigenfunctions on $Y$, so that $\rho^{d_i} \phi_i$ is a homogeneous harmonic function on $C(Y)$ where $d_i$ is the positive root of the indicial equation
	\[
		d_i(d_i+2) = \lambda_i.
	\]
	The homogeneous harmonic functions form a basis for the space of harmonic functions on $C(Y)$ so that a harmonic function, $f = \sum_{d_i \geq 0} \rho^{d_i} \phi_i$.
	
	If $f$ is orthogonal to $\mH_{\leq d}(B(0,1))$, then $f = \sum_{d_i > d} \rho^{d_i} \phi_i = \sum_{d_i \geq d_*} \rho^{d_i} \phi_i$, and, by homogeneity and $L^2$-orthogonality,
	\begin{align*}
		\| f \|^2_{B(0,\lambda)} & =  \sum_{d_i\geq d_*}  \| \rho^{d_i} \phi_i \|^2_{B(0,\lambda)} \\
		& = \sum_{d_i\geq d_*}  \lambda^{2d_i} \| \rho^{d_i} \phi_i \|^2 _{B(0,1)} \\
		& \leq \lambda^{2 d_*}\| f \|^2_{B(0,1)}
	\end{align*}
	with equality if and only if $f$ is homogeneous of degree $d_*$.
\end{proof}

We use this to derive the following monotonicity result.

\begin{lemma}[\(\epsilon\)-monotonicity]\label{lem:epsilonmon}
	Let $0<\lambda<1$, and \(\alpha>0\) small enough that no model cone in $\mathcal C$ has indicial roots in \((2, 2+\alpha]\). There is \(\epsilon=\epsilon(\alpha, \lambda) > 0\) such that if \(B(x, \rho)\) is \(\epsilon\)-close to a ball at the apex of a model cone,  then for all $f$ $L^2$-orthogonal to \(\mathcal{H}_{\leq 2}^{\epsilon} (B(x, \rho))\) with \(\|\Delta f\|_{B(x, \rho)}<\epsilon\), we have
	\begin{equation*}
	\|f\|_{B(x, \lambda \rho)} \leq \lambda^{2+\alpha} \|f\|_{B(x, \rho)}.
	\end{equation*}
\end{lemma}

\begin{proof}
	We proceed by contradiction. If no \(\epsilon>0\) exists, then we can find a sequence \(\epsilon_k \to 0\) with $B(x, \rho_k)$ $\epsilon_k$-close to a ball at the apex of $C(Y)$ for some fixed model cone $C(Y)$, together with functions \(f_k\) $L^2$-orthogonal to \(\mathcal{H}_{\leq 2}^{\epsilon} (B(x_k, \rho_k))\) with \(\|\Delta f_k\|_{B(x, \rho_k)} < \epsilon_k\), but such that
	\begin{equation*}
	\|f_k\|_{B(x, \lambda \rho_k)} > \lambda^{2+\alpha} \|f_k\|_{B(x, \rho_k)} .
	\end{equation*}
	We can assume that \(\|f_k\|_{B(x, \rho_k)}=1\) for all \(k\).
	
	As $k\to \infty$, $\rho_k^{-1} B(x,\rho_k)$ converges in Gromov-Hausdorff distance to $B(0,1) \subset C(Y)$. By the gradient estimate (Caccioppoli inequality) and Rellich compactness, the \(f_k\) converge IN $L^2$ to a harmonic function \(f\) on \(B(0, 1)\) with $\|f\|_{B(0,1)}=1$, see \cite[Lemma 1.3]{ding}, which is \(L^2\)-orthogonal to \(\mathcal{H}_{\leq 2}(B(0,1))\) and
	\[
		\|f\|_{B(0, \lambda)} \geq \lambda^{2+\alpha} \|f\|_{B(0, 1)}.
	\]
	But, by the spectral decomposition lemma, we conclude 
	\[
		\lambda^{2+\alpha} \|f\|_{B(0, 1)} \leq \|f\|_{B(0, \lambda)} \leq \lambda^{d_*} \|f\|_{B(0, 1)}
	\]
	where $d_*$ is the smallest indicial root for $C(Y)$ greater than $d=2$. This is only possible if $2 \leq d_* \leq 2+\alpha$, which contradicts the choice of $\alpha$.
\end{proof}

\subsection{Preliminary estimates}
We establish a \(C^{1, \alpha}\) bound and use it to prove an interior estimate for the \(L^2\)-norm of \(\dd u\).

\begin{proposition} \label{prop:C1alpha}
	Let \(u \in W^{1,2}\) be such that \(\Delta u = f\) on \(B_4\) with \(f \in C^{\alpha}\). Then, for every \(x \in B_{1/2}\) there is \(\tau \in \Lambda^{1,0}(\C^2)\) such that
	\begin{equation*}
	\| \p u - \tau \|_{B_{\rho}(x)} \leq C \rho^{\alpha}
	\end{equation*}
	for all \(0<\rho<1/2\).
\end{proposition}
Here we mean \(\tau=\tau_i \eta_i\) with \(\eta_1, \eta_2\) an orthonormal coframe of \((1,0)\)-forms, and similarly for \(\p u\).
We will be brief because the same line of argument applies later on the \(C^{2, \alpha}\) bound (see Proposition ~\ref{prop:SCHEST}).
\begin{proof}
	We iterate the following one step improvement result:
	\begin{itemize}
		\item[\(\dagger\)] Let \(\alpha>0\) so that there are no indicial roots in the interval \((1, 1+\alpha]\) for all model cones. Then there are \(0 < \lambda < 1\), \(\epsilon>0\) and \(\delta>0\) such that the following holds: Let
		\(\Delta u = f\) on \(B(x, \lambda^k)\) with \(\|u\|_{B_{\lambda^k}(x)} \leq 1\), \(\|f\|_{B_{\lambda^k}(x)} < \delta \) and \(B(x, \lambda^k)\) \(\epsilon\)-close to a ball at the apex of a model cone. Then there is \(P \in \mathcal{H}_{\leq 2}^{\epsilon}(B(x, \lambda^k))\) (indeed \(P \in \mH_{\leq 1}^{\epsilon}\)) such that
		\begin{equation*}
		\|u - P \|_{B(x, \lambda^{k+1})} \leq \lambda^{1+\alpha} .
		\end{equation*}
		Moreover, \(\|P\|_{B(x,\lambda^k)} \leq C\) for some uniform constant \(C\).
	\end{itemize}
	To prove \(\dagger\) we fix \(\epsilon\) as in the monotonicity lemma. We take a harmonic approximation \(\|u-h\|<\mu\) and write \(h= h_{\leq 1} + h_{>1}\) with \(P:= h_{\leq 1} \in \mathcal{H}_{\leq 1}^{\epsilon}\) the \(L^2\)-projection. Write \(\|u-P\|<\|u-h\|+\|h_{>1}\|\) and adjust \(\mu\), \(\lambda\) to get the estimate.
	
	Now, fix \(\epsilon>0\), \(\delta>0\) and \(0<\lambda<1\) as in \(\dagger\). We can clearly assume that \(\|u\|_B \leq 1\) and \(|f|_{C^{\alpha}(x)} < \delta\). We have a controlled number \(N(\epsilon, \lambda)\) of bad scales, in the sense that \(B(x, \lambda^k)\) is \(\epsilon\)-close to a ball on a model cone for all \(k\) except at most \(N\). Set \(u_0=u\) and \(u_k = u_{k-1} -P_k\) with \(P_k\) given by \(\dagger\) applied to \(\lambda^{-(k-1)(1+\alpha)}u_{k-1}\) if \(\lambda^k\) is a good scale and \(P_k=0\) otherwise. This way \(\|u_k\|_{B(x, \lambda^k)} \leq C \lambda^{(1+\alpha)k}\). Set \(\tau_k= \lambda^{k(1+\alpha)} \p P_k(x)\) and \(\tau= \sum_k \tau_k\). Use the scaled Caccioppoli inequality
	\[\| \p u_k \|_{B(x, \lambda^{k+1})} \leq C (\lambda^k\|\Delta u_k\|_{B(x, \lambda^k)} + \lambda^{-k}\|u_k\|_{B(x, \lambda^k)})\] 
	to deduce
	\[\|\p u-\tau\|_{B(x, \lambda^k)}\leq C \lambda^{k\alpha}\]
	for all \(k\). 
\end{proof}

Next, we want a bound on \(\|\dd u\|_{L^2}\). A standard argument to bound the \(L^2\)-norm of the Hessian is to use Bochner formula and integrate by parts. However, we do not have control on the Ricci curvature of the approximate solution, only on its Ricci potential. Fortunately, the Bochner formula on a K\"ahler manifold decomposes into two parts and the part corresponding to \(|\dd u|^2\) does not involve Ricci; see \cite[Proposition 1]{Liu}.

\begin{proposition}
	Assume \(u\) is in \(W^{1,2}\) and \(\Delta u=f \in C^{\alpha}\), then
	\begin{equation}\label{eq:L2bound}
		\| \dd u \|_{B(x, \rho/2)} \leq C \left(\rho^{-2}\|u\|_{B(x, \rho)} + \|f\|_{B(x, \rho)}\right) .
	\end{equation}
\end{proposition}

\begin{proof}
	We work on the rescaled unit ball \(B=\rho^{-1}B(x, \rho)\). Without loss of generality we can assume that \(u\) has compact support contained in \(B\).
	In the complement of the conically singular set we have (up to dimensional factors)
	\[ \dd u \wedge \dd u = \left( |\dd u|^2 - (\Delta u)^2 \right) \omega^2 .\]
	Let \(\chi_{\epsilon}\) be a cut-off equal to \(1\) outside the \(\epsilon\)-tube around the lines and vanishing on the singular set. We multiply by \(\chi_{\epsilon}\) and integrate by parts to get
	\begin{align*}
	\int_B \chi_{\epsilon} \left( |\dd u|^2 - (\Delta u)^2 \right) \omega^2 &= \int_B \dd \chi_{\epsilon} \wedge \p u \wedge \bar{\p} u \\
	& \leq C \int_B |\dd \chi_{\epsilon}| .
	\end{align*}
	In the inequality we have used that the gradient of \(u\) is uniformly bounded, as provided by Proposition \ref{prop:C1alpha}. We take \(\chi_{\epsilon}\) such that \(\int |\dd \chi_{\epsilon}| \to 0\) as \(\epsilon \to 0\). The construction of cut-off functions with \(\int |\Delta \chi_{\epsilon}| \to 0\) as \(\epsilon \to 0\) is standard in the presence of real codimension two singularities, see \cite[Lemma 5.3]{BKMR} - the same construction applies in our setting because we have uniform control on tubular neighbourhoods of the singular set. Since \(\chi_{\epsilon}\) can be taken to depend only in the radial direction transverse to the conical divisor, we may also assume that \(\dd \chi_{\epsilon}\) has rank one almost everywhere. It follows that \(|\Delta \chi_{\epsilon}| \sim |\dd \chi_{\epsilon}|\) so \(\lim_{\epsilon \to 0}\int |\dd \chi_{\epsilon}| =0\) and we are done. 
\end{proof}

\subsection{Main result}

In this section we establish our main estimate, Proposition \ref{prop:SCHEST}. We prove it by a perturbation method, along the lines of \cite{dBoEdw}, which has the advantage of being robust and (contrary to \cite{Donaldson}) does not appeal to symmetry. For a classical reference on this perturbation technique see \cite{caffarelli}. Similarly, we expect that non-linear analogues of the Schauder estimate in Proposition \ref{prop:SCHEST} hold true for the complex Monge-Amp\`ere equation.

Write \(B_r\) for balls centred at the origin in \((\C^2, \omega)\). 
The main result of this section is the following

\begin{proposition} \label{prop:SCHEST}
	Fix \(0 < \alpha < 1/\beta_3 -1\). 
	Let \(u \in W^{1,2}\) be such that \(\Delta u = f\) on \(B_4\) with \(f \in C^{\alpha}\). Then, for every \(x \in B_{1/2}\) there is \(\tau \in \Lambda^{1,1}(\C^2)\) such that
	\begin{equation*}
	\| \dd u - \tau \|_{B_{\rho}(x)} \leq C \rho^{\alpha}
	\end{equation*}
	for every \(0<\rho<1/2\).
\end{proposition}

Same as before, in the integral estimate we mean \(\tau=\tau_{i\bar{j}} \eta_i \bar{\eta}_j\) with \(\eta_1, \eta_2\) an orthonormal coframe of \((1,0)\)-forms for \((\C^2, \omega)\), and similarly for \(\p u\).
The key for proving Proposition \ref{prop:SCHEST} is the following one step improvement lemma.

\begin{lemma} \label{lem:ONESTEP}
	Let \(0 < \alpha < 1/\beta_3 -1\). There are \(0 < \lambda < 1\), \(\epsilon>0\) and \(\delta>0\) such that the following holds: Let
	\(\Delta u = f\) on \(B(x, \lambda^k)\) with \(\|u\|_{B_{\lambda^k}(x)} \leq 1\), \(\|f\|_{B_{\lambda^k}(x)} < \delta \) and \(B(x, \lambda^k)\) \(\epsilon\)-close to a ball at the apex of a model cone. Then there is \(P \in \mathcal{H}_{\leq 2}^{\epsilon}(B(x, \lambda^k))\) such that
	\begin{equation*}
	\|u - P \|_{B(x, \lambda^{k+1})} \leq \lambda^{2+\alpha} .
	\end{equation*}
	Moreover, \(\|P\|_{B(x,\lambda^k)} \leq C\) for some uniform constant \(C\).
\end{lemma}

\begin{proof}
	Let \(\epsilon\) be chosen so that the \(\epsilon\)-monotonicity Lemma \ref{lem:epsilonmon} applies with \(0 < \alpha' < 1/\beta_3 -1\) for some \(\alpha' > \alpha\) and some \(0 < \lambda< 1\) to be determined later. Take \(\mu>0\) to be fixed later and let
	\(\delta>0\) small so that the harmonic approximation Lemma \ref{lem:HARMAPROX} applies to give
	 \(h\) harmonic on \(B(x, \lambda^k)\) with \(\|u-h\|_{B(x, \lambda^k)} < \mu\). We can further assume \(\|h\|_{B(x, \lambda^k)} \leq 2\). Write \(h = h_{\leq 2} + h_{>2}\) with \(P = h_{\leq 2} \in \mathcal{H}_{\leq 2}^{(\epsilon)}(B(x, \lambda^k))\) the \(L^2\)-orthogonal projection. We get that
	\begin{align*}
	\|u - P\|_{B(x, \lambda^{k+1})} &\leq \|u-h\|_{B(x, \lambda^{k+1})} + \|h_{>2}\|_{B(x, \lambda^{k+1})}  \\
	&\leq C\lambda^{-4}\|u-h\|_{B(x, \lambda^k)} + \lambda^{2+\alpha'} \|h_{>2}\|_{B(x, \lambda^k)} \\
	&\leq C_1 \lambda^{-4} \mu + C_2 \lambda^{2+\alpha'}  .
	\end{align*}
	We take \(0<\lambda<1\) such that \(C_2 \lambda^{2+\alpha'} < \lambda^{2+\alpha}/2\) and \(\mu\) small so that \(C_1 \lambda^{-4} \mu < \lambda^{2+\alpha}/2\); therefore \(\|u - P\|_{B(x, \lambda^{k+1})} < \lambda^{2+\alpha}\).
\end{proof}

We proceed with the proof of Proposition \ref{prop:SCHEST}.

\begin{proof}
	Take \(\lambda, \epsilon, \delta\) as in Lemma \ref{lem:ONESTEP}.
	Dividing by \( 1 + \|u\|_{B_4} + \frac{1}{\delta}\|f\|_{C^{\alpha}}\), we can assume that \(\|u\|_{B_4} < 1\) and \(\|f\|_{C^{\alpha}} < \delta\). Fixing $x \in B_{1/2}$ and subtracting \(f(x)\psi\) from \(u\), we reduce to \(f(x)=0\). Here we recall that \(\psi\) is the potential for the approximate solution \(\omega=i\dd\psi\).

	Let \(u_0 = u\). We set \(u_k = u_{k-1}- \lambda^{(2+\alpha)(k-1)}P_k\) with \(P_k\) given by Lemma \ref{lem:ONESTEP} applied to \(\lambda^{-(2+\alpha)(k-1)}u_{k-1}\) for good scales \(\lambda^k\) and \(P_k=0\) otherwise. This way
	\begin{equation*}
	\|u_k\|_{B(x, \lambda^k)}\leq C \lambda^{k(2+\alpha)} 	
	\end{equation*}
	for all \(k \geq 0\). Let \(\tau = \sum_j \tau_j\) with \(\tau_j = \lambda^{(2+\alpha)(j-1)} \dd P_j(x)\), so \(|\tau_j| < C \lambda^{j\alpha}\) and
	\begin{equation} \label{eq:MAINEST1}
		\| \dd u - \tau\|_{B(x, \lambda^k)} \leq \|\dd u_k\|_{B(x, \lambda^k)} + \| \sum_{j=1}^{k} \lambda^{(2+\alpha)(j-1)} \dd P_j - \tau\|_{B(x, \lambda^k)} .
	\end{equation}
	We estimate the first term using equation \eqref{eq:L2bound}
	\begin{align*}
		\| \dd u_k \|_{B(x, \lambda^k)} &\leq C \left( \lambda^{-2k}\|u_k\|_{B(x, \lambda^{k-1})} + \|\Delta u_k\|_{B(x, \lambda^{k-1})} \right) \\
		&\leq C \lambda^{k\alpha} .
	\end{align*}
	Indeed, \(\Delta u_k = f - \sum_{j=1}^{k}\lambda^{(2+\alpha)(j-1)} \Delta P_j\). We have uniform control on the \(C^{\alpha'}\) norm of \(\lambda^{(2+\alpha')j}\Delta P_j\) for some fixed \(\alpha'>\alpha\) and  \(\Delta P_j (x)=0\). We conclude
	\begin{align*}
		\sum_{j=1}^{k}\lambda^{(2+\alpha)j} \|\Delta P_j\|_{B(x, \lambda^k)} &= \sum_{j=1}^{k}\lambda^{(\alpha-\alpha')j} \|\lambda^{(2+\alpha')j} \Delta P_j\|_{B(x, \lambda^k)} \\
		&\leq C \sum_{j=1}^{k} \lambda^{(\alpha-\alpha')j} \lambda^{k\alpha'} \\
		&= C \lambda^{k\alpha} \sum_{j=1}^{k} \lambda^{(\alpha'-\alpha)(k-j)} .
	\end{align*}
	
	In the same way, we use the \(C^{\alpha'}\) bound on \(\lambda^{(2+\alpha')j}\dd P_j\) to handle the second term on the r.h.s. of \eqref{eq:MAINEST1} as follows 
	\begin{align*}
		\|\sum_{j=1}^{k} \lambda^{(2+\alpha)(j-1)} \dd P_j - \tau\|_{B(x, \lambda^k)} &\leq \sum_{j=1}^{k} \lambda^{(2+\alpha)(j-1)} \|\dd P_j - \dd P_j(x)\|_{B(x, \lambda^k)} \\ &+  \sum_{j=k+1}^{\infty} |\tau_j| \\
		&\leq C \lambda^{k\alpha} .
	\end{align*}
	We conclude from equation \eqref{eq:MAINEST1} that \(\|\dd u - \tau \|_{B(x, \lambda^k)} \leq C \lambda^{k\alpha}\).
\end{proof}

\section{Perturbation to a Ricci-flat metric} \label{sect:PERTURBATION}

Establishing Proposition \ref{prop:SCHEST} is the essential result needed to use the perturbation method. Still, since the real parts of holomorphic functions are harmonic with respect to any K\"ahler metric, the kernel of the Laplace operator of our approximate solution metric on a ball is infinite dimensional. To reduce to finite dimensions, we compactify.\footnote{An alternative would be to introduce boundary conditions.}
 
\subsection{Fredholm setup}
We have a natural inclusion \(\C^2 \subset \mathbf{CP}^2\). We still denote by \(L_i\) the corresponding complex projective lines.

\begin{lemma} \label{lem:EXT}
	There is a K\"ahler metric on \(\mathbf{CP}^2\) that agrees with the approximate solution in a neighbourhood of the  intersection point of the lines and has standard cone singularities of angle \(2\pi\beta_i\) elsewhere along \(L_i\).
\end{lemma}

\begin{proof}
	Let \(h\) be a smooth Hermitian metric on \( \mathcal{O}(1) \) and \( \ell_j \) be  holomorphic sections of \( \mathcal{O}(1) \) with \( \ell_j^{-1}(0) = L_j \) and let \(\eta\) be a smooth K\"ahler metric on \(\mathbf{CP}^2\). 
	We claim that 
	\begin{equation} \label{eq:EXTLEM0}
	i\dd (|\ell_1|_h^{2\beta_1} |\ell_2|_h^{2\beta_2} |\ell_3|_h^{2\beta_3}) > - C \eta
	\end{equation}	
		for some uniform \(C>0\).
	Indeed this follows from the identity
	\[ i \dd u = u i \dd \log u + u^{-1} i \partial  u \wedge \overline{\partial} u \geq u i \dd \log u  \]
	with \( u = |\ell_1|_{h}^{2\beta_1} |\ell_2|_h^{2\beta_2} |\ell_3|_h^{2\beta_3}\) and noticing that each \(i \dd \log |\ell_j|_h^2\) is a smooth form on the complement of \(L_j\) which admits a smooth extension to \(\mathbf{CP}^2\) (and therefore is bounded).

	Let \(\chi\) be a standard cut-off function with \(\chi = 0 \) on \(B_1\) and \(\chi = 1 \) outside \(B_2\). Write the approximate solution as \(\omega=i\dd\psi\) on \(B_3\). We claim that there is some \(C>0\) such that 
	\begin{equation} \label{eq:EXTLEM}
	i \dd ( (1-\chi) \psi + \chi |\ell_1|_h^{2\beta_1} |\ell_2|_h^{2\beta_2} |\ell_3|_h^{2\beta_3} ) > - C \eta .
	\end{equation}
	Indeed, it follows from \eqref{eq:EXTLEM0} that we only need to check \eqref{eq:EXTLEM} on  \( B_2 \setminus B_1 \). On the other hand at each point \(p\) of \(L_j\) lying on \(B_2 \setminus B_1\) we can find complex coordinates which do not meet the other lines and such that \(L_j = \{z_1 = 0\} \). In such coordinates
	\[(1-\chi) \psi + \chi |\ell_1|_h^{2\beta_1} |\ell_2|_h^{2\beta_2} |\ell_3|_h^{2\beta_3} = F_1 |z_1|^{2\beta_j} + F_2 \]
	 where \(F_1, F_2\) are smooth  functions and \(F_1\) is uniformly bounded below by a positive constant in a neighbourhood of \(p\). The inequality \eqref{eq:EXTLEM} follows from 
	 \begin{align*}
	 	i \dd \left(F_1|z_1|^{2\beta_j}\right) &\geq \left(F_1|z_1|^{2\beta_j}\right) i \dd \log \left(|z_1|^{2\beta_j}F_1\right) \\
	 	&= \left(F_1|z_1|^{2\beta_j}\right) i \dd \log \left(F_1\right) .
	 \end{align*}
	 
	Let \(\eta_0\) be a smooth non-negative \((1, 1)\)-form \(\eta_0 \geq 0\) such that \(\eta_0 = \eta\) on \( \mathbf{CP}^2 \setminus B_{1/2} \) and \(\eta_0 = 0 \) on \(B_{\epsilon}\) for some $\epsilon >0$. (The construction of such \((1,1)\)-forms is standard.)
	For \(\delta>0\) we set 
	\begin{equation} \label{eq:EXTMET}
	\omega = \eta_0 + \delta i \dd \left( (1-\chi) \psi + \chi |\ell_1|_h^{2\beta_1} |\ell_2|_h^{2\beta_2} |\ell_3|_h^{2\beta_3} \right) .
	\end{equation}
	It is then clear from Equation \eqref{eq:EXTLEM} that if \(\delta>0\) is sufficiently small, then \(\omega\) satisfies the requirements of the lemma.
\end{proof}

Consider \((\mathbf{CP}^2, \omega)\) as in Lemma \ref{lem:EXT}. Set
\[ [f]_{C^{\alpha}(x)} = \sup_{0<\rho<1} \rho^{-\alpha}\Big\|f - \fint_{B_\rho(x)} f \Big \|_{B_{\rho}(x)}.\]
Here we use the Campanato criterion of H\"older functions, see \cite[Section 1.1]{ACM} where Campanato and Morrey spaces are used in the context of metric measure spaces. The \(C^{2, \alpha}\) norms are defined by looking at the components of \(\p f\) and \(\dd f\) with respect to an orthonormal coframe. Proposition \ref{prop:SCHEST} gives us the following global result.

\begin{corollary}
	If \(u\) is a weak solution of \(\Delta u = f\) and \(f \in C^{\alpha}\), then \(u \in C^{2, \alpha}\) and 
	\[\|u\|_{C^{2, \alpha}} \leq C_0 (\|f\|_{C^{\alpha}} + \|u\|_{C^0}) .\]
\end{corollary}

It follows from the Poincar\'e inequality that if \(\int f = 0\) then there is a  weak solution to \(\Delta u =f\). The solution is unique up to addition of a constant. We conclude that \(\Delta\) is an isomorphism between the  zero average \(C^{2, \alpha}\) and the zero average \(C^{\alpha}\) spaces. The inverse of the Laplacian has bounded norm, as shown by the following

\begin{corollary}
	There is \(C>0\) such that 
	\[ \|u\|_{C^{2, \alpha}} \leq C \|\Delta u\|_{C^{\alpha}} \]
	for every \(u\) such that \(\int u = 0\). (We can also replace the zero average condition  by requiring that \(u(p)=0\) for some fixed point \(p\).)
\end{corollary}

\begin{proof}
	Otherwise we get a sequence \(\|u_k\|_{2, \alpha} =1\), \(\|\Delta u_k\|_{\alpha} \leq 1/k\).  It follows from the interior Schauder estimates that \(\|u_k\|_0 \geq 1/C_0 - \|\Delta u_k\|_{\alpha} \geq 1/(2C_0)\) for \(k\) large, we let \(|u_k(x_k)| \geq 1/(2C_0)\). Up to a subsequence \(u_k \to u_{\infty}\) with \(u_{\infty}\) harmonic, hence constant. Since \(\int u_{\infty}=0\) (or \(u_{\infty}(p)=0\)) we conclude that \(u_{\infty}=0\). But this contradicts \(|u_{\infty}(x_{\infty})| \geq 1/(2 C_0)\).
\end{proof}

\subsection{Implicit function theorem}

For our approximate solution \((\mathbf{CP}^2, \omega)\) we have
\[\omega^2 = e^{-h} \Omega \wedge \bar{\Omega} ,\]
with \(h \in C^{\alpha}\), \(h(0)=0\). Here, we have also chosen some extension of \(\Omega \wedge \bar{\Omega}\) with \(\int \Omega \wedge \bar{\Omega}=\int \omega^2\).
Our goal is to solve, for \(\omega_u := \omega + i \dd u\), the equation
\begin{equation} \label{eq:MA}
	\omega_u^2 = e^{h} \omega^2
\end{equation}
in a neighbourhood of \(0\).
 
Consider the Monge-Amp\`ere operator \(MA(u) = \log (\omega_u^2/\omega^2) \),
\[MA:  \mathcal{U} = \{u \in C_0^{2, \alpha}, \,\ \omega_u >0\} \to \mathcal{V}= \{f \in C^{ \alpha}, {} \int (e^f -1) =0\} . \]
Here, the sub-index \(0\) means zero average; so \(\mathcal{U}\) is an open subset of a Banach vector space and \(\mathcal{V}\) is a Banach manifold i.e. a \(C^1\)-hypersurface on the space of \(C^{\alpha}\) functions.
Clearly, \(MA(0)=0\) and \(D(MA)(0) = \Delta\) with
\begin{align*}
	\Delta : T_0 \mathcal{U} \cong C^{2, \alpha}_0 &\to T_0 \mathcal{V} \cong  C^{\alpha}_0 
\end{align*}
an isomorphism.

\begin{proposition} \label{prop:INVTHM}
	Equation \eqref{eq:MA} admits a solution \(u \in C^{2, \alpha}\). Moreover, by shrinking the neighbourhood in which \eqref{eq:MA} is satisfied, we can ensure that \(\|u\|_{2, \alpha}\) is as small as we please.
\end{proposition}

\begin{proof}
	By the implicit function theorem, we can solve 
	\[\omega_{u}^2 = e^{\tilde{h}}\omega^2  \]
	whenever \(\tilde{h} \in \mathcal{V}\) and \(\|\tilde{h}\|_{\alpha}<\mu_0\) for some fixed \(\mu_0>0\).
	Take a sequence of cut-off functions \(\chi_k\) that are equal to \(1\) in \(B_{1/k}(0)\) and equal to \(0\) outside \(B_{2/k}(0)\). Since \(h(0)=0\), it follows that \(\| \chi_k h \|_{\alpha} \to 0\) as \(k \to \infty\). 
	Write \(V= \int \omega^2 = \int \Omega \wedge \bar{\Omega}\)
	and let 
	\[v_k = V- \int e^{\chi_kh} . \]
	In particular, \(v_k \to 0\) as \(k \to \infty\).
	
	Fix a point \(p \in \mathbf{CP}^2\) far from the origin and take a sequence of bump functions \(\chi_{p,k}\) supported in a ball \(B_p\) centred at \(p\) which does not intersect \(B_3\), say, and which satisfy \(\int_{B_p} e^{\chi_{p,k}} =\vol(B_p)+v_k\). Moreover, we can choose \(\chi_{p,k}\) so that they converge smoothly to zero as \(k \to \infty\). 
	
	We set
	\begin{equation*}
		h_k = \chi_k h + \chi_{p,k} .
	\end{equation*}
	Since the supports of \(\chi_k\) and \(\chi_{p,k}\) do not overlap, we have
	\begin{align*}
		\int_{\mathbf{CP}^2} e^{h_k} 
		& = \int_{\mathbf{CP}^2 \setminus B_p} e^{\chi_k h} + \int_{B_p} e^{\chi_{p,k}} \\
		&= \int_{\mathbf{CP}^2} e^{\chi_k h} - \vol(B_p) + \int_{B_p} e^{\chi_{p,k}} \\
		&= \int_{\mathbf{CP}^2} e^{\chi_k h} + v_k = V .
	\end{align*}
	It follows that \(h_k \in \mathcal{V}\) and \(\|h_k\|_{\alpha} < \mu_0\) once \(k\) is sufficiently large. We can therefore solve \(MA(u_k)=h_k\), so in particular \(MA(u_k)=h\) on \(B_{1/k}\).
\end{proof}

To conclude the proof of our main Theorem \ref{MAINTTHM} we set \(\omega_{CY}=\omega_u\) with \(\omega_u\) as in Proposition \ref{prop:INVTHM}. Since \(u \in C^{2, \alpha}\), it is clear that the tangent cone of \(\omega_{CY}\) at the origin is unique and isometric to \(\cone\).

\section{General Picture} \label{sect:GENPICT}

\subsection{Stability}
We fix three or more complex lines \(L_1, \ldots, L_d\) going through the origin and discuss the expected behaviour of K\"ahler-Einstein metrics as the cone angles vary.

The geometric interpretation of \eqref{anglecond} can be understood as the strict violation of the Troyanov condition for the existence of spherical metrics, i.e. positive constant curvature metrics on $\mathbf{CP}^1$ with prescribed conical singularities:

More precisely, if \(0 < \beta_1 \leq \beta_2 \leq \ldots \leq \beta_d < 1\) satisfy the following two conditions:
\begin{equation} \label{eq:KLT}
		\sum_{j=1}^{d} (1-\beta_j) < 2, \tag{K}
\end{equation}
and
\begin{equation} \label{eq:TROY}
		(1-\beta_1) < \sum_{j=2}^{d} (1-\beta_j), \tag{S}
\end{equation}
then there is a unique spherical metric on \(\mathbf{CP}^1\) with cone angle \(2\pi\beta_j\) at each \(L_j \in \mathbf{CP}^1\) \cite{LuTi92,tr91}. Since we are assuming that \(d \geq 3\), these two conditions are also necessary for the existence of spherical metrics. In classical differential geometric terms, \eqref{eq:KLT} comes from Gauss-Bonnet and \eqref{eq:TROY} is the Troyanov condition.

The two conditions have different flavours and can also be interpreted in terms of algebro-geometric properties of the pair \((\C^2, \sum_{j=1}^{d}(1-\beta_j)L_j)\): Equation \eqref{eq:KLT} corresponds to the KLT property while Equation \eqref{eq:TROY} corresponds to stability.

It is convenient to define the weights \(\mu_j = 1-\beta_j\). The pair \((\C^2, \sum_{j=1}^{d}\mu_j L_j)\) is KLT provided that \(\prod_{j=1}^{d}|\ell_j|^{-2\mu_j}\) is locally integrable around \(0\) in the standard Lebesgue sense. Note that \(\prod_{j=1}^{d}|\ell_j|^{-2\mu_j}\) is homogeneous and locally integrable away from the origin because \(\mu_j<1\). By taking spherical coordinates, the KLT condition amounts to
\begin{equation*}
	\int_0^1 t^{3-2\sum \mu_j} dt < \infty 
\end{equation*}
which is equivalent to \eqref{eq:KLT}. On the other hand, Equation \eqref{eq:TROY} is equivalent to \(\mu_i < \sum_{j \neq i} \mu_j\) for every \(i=1, \ldots, d\). This is the standard GIT stability for a weighted configuration of points in the Riemann sphere; and agrees with the `log K-stability' of the affine pair \((\C^2, \sum_j \mu_j L_j)\) polarized by dilations.

Our assumption on the cone angles \eqref{anglecond} automatically implies the KLT property, and, in terms of stability, means that we are in the strictly unstable case. The tangent cone in Theorem \ref{MAINTTHM} agrees with the one predicted algebraically by the theory of normalized volumes of valuations, and there is a destabilizing test configuration from \((\C^2, \sum_{j=1}^{d}\mu_j L_j)\) to \(\cone\) in the central fibre, see \cite[Section 4.1]{Li}.

We introduce the (KLT) open convex polytope 
\[ \mathcal{K} = \{ 0<\mu_j <1, \hspace{2mm} \sum_j \mu_j < 2 \} \subset \R^d ,\] 
a hypercube with its corner \((1,\ldots,1)\) chopped off. The  stable open subset \(\mathcal{S} \subset \mathcal{K}\) is cut out by \(d\) equations \( \cap_{i=1}^d \{ \mu_i < \sum_{j \neq i} \mu_j\}\). There are \(d\) hyperplane walls  \(\mathcal{W} = \cup_{i=1}^d\{\mu_i = \sum_{j\neq i} \mu_j\}\) that form the semistable locus. The strictly unstable region, \(\mathcal{U}\), the complement of \(\overline{\mathcal{S}}\), is open and has \(d\) components. We get a decomposition 
\[\mathcal{K} = \mathcal{S} \sqcup \mathcal{W} \sqcup \mathcal{U}\] 
into \(d+1\) open convex polytopes and \(d\) interior walls.
If \(\mu \in \mathcal{S}\) then there is a polyhedral K\"ahler cone metric with cone angles \(2\pi\beta_j\) along the complex lines \(L_j\), that comes as a lift of the spherical metric on the Riemann sphere and models the local behaviour of K\"ahler-Einstein metrics with these cone angles in general. On the other hand, Theorem \ref{MAINTTHM} provides models for the strictly \emph{unstable} region \(\mu \in \mathcal{U}\). It is an open problem to analyse in differential-geometric terms the case of equality in equation \eqref{anglecond}, that is the semi-stable regime \(\mu \in \mathcal{W}\).

While $\mathcal K$ parametrizes the set of KLT pairs, the boundary \(\p \mathcal{K}\) parametrizes log canonical pairs. It would be interesting to provide models for the asymptotic cuspidal behaviour of the corresponding K\"ahler-Einstein metrics for these pairs.

\subsection{Higher dimensions}
Of course one can take products of our Ricci flat metrics with a flat Euclidean factor to model convergence with multiplicity of conical divisors. We present here a different situation in which an irreducible hypersurface develops a normal crossing in the tangent cone limit.

Let \(\{x_1x_2=1\} \subset \C^2\). Recall Donaldson's Ricci flat model metrics \cite{Donaldson}, which solve
\begin{equation*}
(i \dd \phi)^2 = |1 - x_1 x_2|^{2\beta-2} dx_1 dx_2 \overline{dx_1 dx_2} .
\end{equation*}
Consider the two dimensional \(A_{p-1}\)-singularity given by
\begin{equation*}
D = \{ x_1 x_2 = z^p \} \subset \C^3 .
\end{equation*}
Fix \( \alpha_0 \in (1, p \beta /2)\).
The approximate solution ansatz is then 
\begin{equation*}
\omega = i \dd \left( |z|^2 + \gamma_1 (R \rho^{-\alpha_0}) R^2 + \gamma_2(R \rho^{-\alpha_0}) |z|^{p \beta} \phi (z^{-p/2} \cdot x) \right) .
\end{equation*}
Here \( \lambda \cdot x = (\lambda x_1, \lambda x_2)\) and \(\phi (z^{-p/2} \cdot x)\) is well defined because \(\phi(x)=\phi(-x)\). We have used  \(\rho^2 = |z|^2 + R^2\) and \(R^2 = |x_1|^{2\beta} + |x_2|^{2\beta}\), so \(\omega_{\C_{\beta} \times \C_{\beta}} = \frac{i}{2} \dd R^2\) is the tangent cone at infinity of \(i\dd \phi\).

Along the \(z\) axis, we have \( d(\cdot, 0) \approx |z|\) and the metric, in transverse directions is given by \(|z|^{p\beta} (i \dd \varphi)\). If \(p \beta > 2\) then the smoothing model \(i \dd \phi\) contracts faster than linearly. We expect we can perturb the approximate solution to a Ricci flat metric:

\emph{Assume that \(p \geq 3\) and \(2/p < \beta < 1\). Then there should be a Calabi-Yau metric \(\omega_{CY}\) in a neighbourhood of the origin in \(\C^3\) with cone angle \(2\pi\beta\) along \(D \setminus\{0\}\) and tangent cone at the origin equal to \(\C \times \C_{\beta} \times \C_{\beta}\).}



\bibliographystyle{amsplain}
\bibliography{bib3lines}

\end{document}